\documentclass[a4paper,12pt,oneside]{article}
\usepackage[english]{babel}
\usepackage[T1]{fontenc} 
\usepackage[utf8]{inputenc}
\usepackage{amsthm}
\usepackage{bbm}
\usepackage{amsmath}
\usepackage{amssymb}  
\usepackage{indentfirst}
\usepackage{fancyhdr}
\usepackage{amsthm}
\usepackage{graphicx}
\usepackage{pdfpages}
\usepackage{esint}
\usepackage{url}

\usepackage[numbers,sort&compress]{natbib}

\numberwithin{equation}{section}

\theoremstyle{plain} 
\newtheorem{thm}{Theorem}[section] 
 
\newtheorem{lem}[thm]{Lemma} 
\newtheorem{prop}[thm]{Proposition} 
\newtheorem{rmk}[thm]{Remark} 
\newtheorem{dfn}[thm]{Definition}

\usepackage{xcolor}                    
\definecolor{custom-blue}{RGB}{0,99,166} 
\usepackage{hyperref}
\hypersetup{colorlinks=true, allcolors=custom-blue}

\usepackage{geometry}
\allowdisplaybreaks[4]

\begin{document}

\author{$\text{\sc{Antonio Giuseppe Grimaldi}}^\clubsuit$ \sc{and} $\text{\sc{Stefania Russo}}^\spadesuit$}

\title{Regularity for minimizers of degenerate, non-autonomous, orthotropic integral functionals}

\maketitle
\maketitle

\begin{abstract}
{We prove the higher differentiability of integer order of  locally bounded minimizers of integral functionals of the form
\begin{equation*}
 \mathcal{F}(u,\Omega):= \,\sum_{i=1}^{n} \dfrac{1}{p_i}\displaystyle   \int_\Omega  \, a_i(x) \lvert u_{x_i} \rvert^{p_i} dx- \int_\Omega \omega(x)u(x)  dx,
\end{equation*}
where the exponents \( p_i \geq 2 \) and the coefficients \( a_i(x) \) satisfy a suitable Sobolev regularity. The main novelty consists in dealing with non-autonomous, anisotropic functionals, which depend also on the solution.}
\end{abstract}

\medskip
\noindent \textbf{Keywords:} Anisotropic growth,  Sobolev
 regularity, Higher differentiability 
\medskip \\
\medskip
\noindent \textbf{MSC 2020:} 35J70, 35B65, 49K20.

\let\thefootnote\relax\footnotetext{
			\small $^{\clubsuit}$Dipartimento di Ingegneria, Università degli Studi di Napoli ``Parthenope'',
Centro Direzionale Isola C4, 80143 Napoli, Italy. E-mail: \textit{antoniogiuseppe.grimaldi@collaboratore.uniparthenope.it}}
\let\thefootnote\relax\footnotetext{
			\small $^{\spadesuit}$Dipartimento di Matematica e Applicazioni ``R. Caccioppoli'', Università degli Studi di Napoli ``Federico II'', Via Cintia, 80126 Napoli,
 Italy. E-mail: \textit{stefania.russo3@unina.it}}

\section{Introduction}
The aim of this paper is to investigate the differentiability properties of locally bounded minimizers of convex functionals, exhibiting an orthotropic structure, depending also on the $x$-variable and on the minimizer itself. More precisely, we consider
\begin{equation}\label{functional}
 \mathcal{F}(u,\Omega):= \, \int_\Omega \left[f(x, Du)-\omega(x)u(x) \right] dx,
\end{equation}
\noindent where $\Omega \subset \mathbb{R}^n$ is an open subset, $n\geq 2$, {$u : \Omega \to \mathbb{R}$}, and $f:\Omega\times\mathbb{R}^{ n}\to [0,+\infty)$ has the following form 
\begin{equation}
    f= f(x, \xi)= \sum_{i=1}^{n} \,     \dfrac{a_i(x)}{p_i} \lvert \xi_i \rvert^{p_i}, \label{integrand}
\end{equation}
with $\xi=(\xi_1,\xi_2,...,\xi_n) \in \mathbb{R}^n$.
Here, we assume that $2\leq p_1  \leq \cdots \leq p_n$  and $a_i(x):\Omega \to \mathbb{R}$ are bounded, measurable coefficients such that, for every $i = 1,...,n$,   $$0< \lambda \leq a_i (x) \leq \Lambda  ,$$ for some positive constants $\lambda$, $ \Lambda$ and for a.e.\ $ x \in \Omega$. 

One can easily check  that there exist positive constants $ L, l$ depending on $\lambda, \Lambda, p_i$ such that
\begin{equation}\label{A2}
    \langle D_\xi f(x, \xi)-D_\xi f(x, \eta),\xi-\eta\rangle  \geq  l \sum_{i=1}^{n}(|\xi_i|^2+|\eta_i |^2)^\frac{p_i -2}{2}|\xi_i - \eta_i |^2 \tag{A1}
\end{equation}
\begin{equation}\label{A3}
     |D_\xi f(x, \xi)-D_\xi f(x, \eta)| \leq  L \sum_{i=1}^{n}(|\xi_i |^2+|\eta_i |^2)^\frac{p_i -2}{2}|\xi_i - \eta_i | ,\tag{A2}
\end{equation}

\noindent for a.e.\ $x \in \Omega$ and all $\xi, \eta \in \mathbb{R}^{ n}$.
Concerning the dependence on the $x$-variable, we assume that $$a_i(x) \in W^{1,r}_{\mathrm{loc}}(\Omega) \qquad \forall i=1, \dots,n ,$$ so that,
denoting by $g(x)=  \underset{i}{\max} \{|Da_i(x)| \}$, we have that $g(x) \in L_{\mathrm{loc}}^r (\Omega)$ and the following inequality
\begin{equation}\label{A4}
    |D_\xi f(x,\xi)-D_\xi f(y, \xi)| \leq |x-y|\,(g(x)+g(y)) \, \sum_{i=1}^{n} |\xi_i|^{p_i-1} \tag{A3}
\end{equation}
\noindent holds for a.e.\ $x,y \in \Omega$ and every $\xi \in \mathbb{R}^{n}$.

For the exponents $p_i$, we shall denote by $\mathbf{p} = (p_1, p_2, \ldots, p_n)$ and by
\begin{equation*}
    W^{1,\mathbf{p}}_{\mathrm{loc}}(\Omega) = \left\{ u \in W^{1,1}_{\mathrm{loc}}(\Omega) : u_{x_i} \in L^{p_i}_{\mathrm{loc}}(\Omega), \, i = 1, \ldots, n \right\},
    \end{equation*}
the  corresponding anisotropic Sobolev space.\\
We denote by $\overline{p}$ the harmonic average of $p_i$ and by $\overline{p}^*$ the Sobolev conjugate exponent of $\overline{p}$, i.e.
 \begin{align}\label{definizionePi}
        \frac{1}{\overline{p}} = \frac{1}{n} \sum_{i=1}^{n} \frac{1}{p_i}, \qquad \overline{p}^* = \begin{cases}\frac{n \overline{p}}{n- \overline{p}} & \quad \text{if} \ \overline{p} <n,\\
        \text{any exponent in } [1,\infty) & \quad \text{if} \  \overline{p} \ge n.
        \end{cases}
    \end{align}
   For an exponent $1 \le s \le \infty$, we will denote by $s'$ the H\"older conjugate of $s$.\\
In the following, we will assume that the function $\omega(x) $ belongs to the Sobolev space  $W^{1,{\frac{\mathbf{p}+2}{\mathbf{p}+1}}}_{\mathrm{loc}}(\Omega)$, with ${\frac{\mathbf{p}+2}{\mathbf{p}+1}} = (\frac{p_1+2}{p_1+1}, \frac{p_2+2}{p_2+1}, \ldots, \frac{p_n+2}{p_n+1})$.\\

\noindent Let us recall the following definition of local minimizer.
\begin{dfn}
{ \em A function {$u\in W_{\rm loc}^{1,\mathbf{p}}(\Omega)$} is a  local minimizer of
\eqref{functional} if, for every open subset $\tilde{\Omega} \Subset \Omega$, we have   
$\mathcal{F}(u;\tilde{\Omega})\le  \mathcal{F}(\varphi;\tilde{\Omega})$ 
for all $\varphi\in \mathcal{C}_0^{\infty}(\tilde{\Omega})$.}
\end{dfn}
The energy densities $f$ that satisfy anisotropic growth condition fit into the wider context of those with $(p,q)$-growth condition as follows
$$|\xi|^p \leq f(x,\xi) \leq (1 + |\xi|^2)^{\frac{q}{2}}, \qquad1<p<q, $$
where $p= c(p_i)$ and $q= \max \{p_i : i=1, \dots, n \}.$\\
The study of regularity properties of local minimizers to functionals with non-standard growth started with the pioneering papers by Marcellini \cite{MarcelliniEs1, MarcelliniEs2}.
 It is now well understood that, in order to obtain some regularity for the local minimizers of functionals with $(p,q)$-growth, $p<q$,  a restriction between the exponents $p$ and $q$ must be imposed (see the counterexamples in \cite{Giaquinta,MarcelliniEs3}).  Indeed, the gap $\frac{q}{p}$ cannot differ to much from $1$ and sufficient conditions to the regularity can be expressed as
$$\frac{q}{p}\le c(n)\displaystyle \to 1\,\,\textnormal{as}\,\,\,\,n\to \infty,$$
according to \cite{MarcelliniEs1,MarcelliniEs2}.

In recent years there has been a considerable of interest in variational problems with $(p,q)$-growth conditions (see, for example, \cite{Adi,Barone,6,12,SuggeritoRefery,Ming1,Ming2,Ming3,Ming4,Ele1,Ele2,Ele3,25,Hasto,Koch,16} and references therein). It is impossible to give a complete overview on the subject, however we refer the interested reader to the recent survey \cite{surveyMar1,surveyMing}.\\
Compared to the above mentioned papers, the difference with our work lies in the fact that the behavior of our functional at \eqref{functional} depends on the components of the gradient and, although they exhibit $(p,q)$-growth, they need to be treated with appropriate arguments.\\
Actually, the regularity of local minimizers of functionals with orthotropic structure as in \eqref{functional} is a well-studied problem: some results can be found in \cite{6,5,LemmaBrasco,Carozza,17,53,russo}, for the homogeneous case \( \omega = 0 \), and for the non-homogeneous one, we refer to  \cite{BraTau,CMM}.
\\
In the autonomous case, i.e.\ $a_i(x) \equiv 1$,  for all $i=1,...,n$, it has been proved in \cite{Brascop12}, for the sub-quadratic case  with $\omega \in L^{1+ \frac{2}{p_1}}_{\mathrm{loc}}(\Omega)$,
that any local minimizer \( u \in W^{1,\mathbf{p}}_{\text{loc}}(\Omega) \cap L^\infty_{\text{loc}}(\Omega) \)  satisfies 
\begin{equation}
    |u_{x_i}|^{\frac{p_i - 2}{2}}  \, u_{x_i} \in W^{1,2}_{\text{loc}}(\Omega), \quad i = 1, \dots,n. \label{diffproperties}
\end{equation}
For the super-quadratic case, Brasco et al.\ \cite{BraTau} consider the following widely degenerate orthotropic functional 
$$F(u; \Omega) = \sum_{i=1}^{n} \int_{\Omega} \frac{1}{p_i} \left( |u_{x_i}| - \delta_i \right)^{p_i} \, dx - \int_{\Omega} \omega \,u \, dx, $$
with $\delta_i \ge 0$, $p_i \ge 2$, $ u \in W^{1,\mathbf{p}}_{\text{loc}}(\Omega)\cap L^\infty_{\text{loc}}(\Omega)$ and 
$ \omega \in W^{1,\mathbf{p'}}_{\text{loc}}(\Omega), $ $\mathbf{p'}=(p_1', \cdots, p_n')$, where $p_i'$ is the H\"older conjugate of $p_i$.
The authors proved the higher differentiability of the local minimizers, that is
\[
\left( |u_{x_i}| - \delta_i \right)^\frac{p_i}{2}_+ \frac{u_{x_i}}{|u_{x_i}|}  \in W^{1,2}_{\text{loc}}(\Omega), \quad i = 1, \dots, n.
\]
In particular, for local weak solutions of the anisotropic orthotropic $p$-Laplace equation (i.e. the case $\delta_i = 0$), they get the regularity at \eqref{diffproperties}.

 In the paper \cite{16} by Marcellini, it is clearly stated that currently there are few known regularity results for problems of this type, in relation to the properties of the coefficients $a_i(x)$. Indeed, as far as we know, few regularity results are available for the local minimizers of functionals with an orthotropic structure in the so-called non-autonomous case (see \cite{feopass,russo}).
More precisely, in the super-quadratic case it has been recently proved in \cite{russo} that, assuming the condition $p_n < p_1+2$ and Sobolev regularity on the coefficients $a_i(x)$ of the type $W^{1,r}$, with $r > p_n+2$, locally bounded minimizers $u$ of the functional \eqref{functional}, with $\omega  \equiv 0$, satisfy the higher differentiability properties at \eqref{diffproperties}.\\
\noindent Here, the main novelty with respect to the paper \cite{russo} is that we allow the presence of a forcing term $\omega$. We are interested in the conditions that must be imposed on the function $\omega$ in order to achieve analogous higher differentiability results for the solutions. \\
 Actually, our main result consists in proving that a suitable Sobolev regularity of the coefficients $a_i(x)$ and of the forcing term $\omega$ transfers into a higher differentiability of integer order for the local minimizers of \eqref{functional}. Moreover, we will require a weaker differentiability assumption on $\omega$ with respect to the one made in \cite{BraTau}. Indeed,  since $\frac{p_i+2}{p_i+1} < p_i'$, we have the embedding $W^{1, \mathbf{p'}}_{\mathrm{loc}}(\Omega) \hookrightarrow W^{1,{\frac{\mathbf{p}+2} {\mathbf{p}+1}}}_{\mathrm{loc}}(\Omega)$.
 
We recall that recent studies have shown that the weak differentiability of the map \( D_\xi f(x,\xi) \), whether of integer or fractional order with respect to the \( x \)-variable, is a sufficient condition to achieve higher differentiability (see \cite{cup, Torricelli, 25,40} for the case of Sobolev spaces with integer order and \cite{ 2,15, Grimaldi1, Ipocoana1, Russo} for the fractional one) both in case of standard and non standard growth. \\
In particular, the general requirement for obtaining higher differentiability for local minimizers is a Sobolev-type regularity for the coefficients with an exponent
$r$ that is greater than or equal to the dimension 
$n$. When dealing with bounded minimizers, the situation changes, and sufficient assumptions on the Sobolev regularity of the coefficients can be made independently of the dimension (see \cite{CKP, Colombo, cup,25}). Our result follows this approach. Even though the order of summability of the coefficients depends on the dimension $n$, it does so only through the bound that allows us to handle locally bounded minimizers.

{In the statement our main result, we shall use the  auxiliary function $V_{p}(\xi)$  defined by
\begin{equation}\label{DefVp}
    V_{p}(\xi):= |\xi|^{\frac{p-2}{2}} \xi,
    \end{equation}
for all $\xi\in \mathbb{R}^{n}$. 

We will denote by
 \begin{align}\label{definizionet}
        \frac{1}{\overline{q}} = \frac{1}{n} \sum_{i=1}^{n} \frac{p_i+1}{p_i+2}, \qquad t: = \begin{cases}\frac{n \overline{q}}{n- \overline{q}} & \quad \text{if} \ \overline{q} <n,\\
        \text{any exponent in } \left(\left(\frac{\overline{p}^*}{p_n}\right)',\infty \right) & \quad \text{if} \  \overline{q} \ge n.
        \end{cases}
    \end{align}

Our aim is to prove the following
\begin{thm}\label{thmBfinito}
Let $\omega \in  W_{\mathrm{loc}}^{1,\frac{\mathbf{p}+2}{\mathbf{p}+1}}(\Omega) $ and let $u \in {W_{\mathrm{loc}}^{1,\mathbf{p}}}(\Omega)$ be a local minimizer of \eqref{functional} under assumptions \eqref{A2}--\eqref{A4}, with exponents $p_i\geq 2, \forall i=1,\dots,n $, such that
\begin{equation}\label{Ipotesip}
p_n <
\begin{cases}
    \min \Big\{ \dfrac{\overline{p}^*}{t'} ,p_1 +2 \Big\} \, &\text{if } \overline{p} < n,  \\
     p_1 +2 \quad &\text{otherwise},
    \end{cases}
\end{equation}
and with a function $g \in {L^{r}_{\mathrm{loc}}(\Omega)}$, where $r$ satisfies the condition
\begin{equation}\label{Ipotesir}
   r> p_n + 2 .
\end{equation} 
Then,
\begin{equation*}
{V_{p_i}(u_{x_i}) \in W^{1,2}_{\mathrm{loc}}(\Omega)} \qquad\forall i=1,\dots,n,
\end{equation*}
and the following estimates
\begin{align*}
    &\sum_{i=1}^{n} \left(  \int_{B_{R/4}} |u_{x_i}|^{p_i +2 } dx \right) \notag\\
    &\leq c  \left( 1+\sum_{i=1}^{n} \left(\Vert u_{x_i} \Vert_{L^{p_i}(B_R)} + {\Vert \omega_{x_i} \Vert_{L^\frac{p_i+2}{p_i+1}(B_{R})}} \right) +\Vert \omega\Vert_{L^t(B_R)}+ \Vert g \Vert_{L^r (B_{R})} \right)^\sigma 
\end{align*}
and
\begin{align*}
        &\sum_{i=1}^{n}  \left( \int_{B_{R/4}} |u_{x_i x_j}|^2|u_{x_i}|^{p_i-2}dx \right) \notag\\
         &\leq c  \left( 1+\sum_{i=1}^{n} \left(\Vert u_{x_i} \Vert_{L^{p_i}(B_R)} + {\Vert \omega_{x_i} \Vert_{L^\frac{p_i+2}{p_i+1}(B_{R})}} \right) +\Vert \omega\Vert_{L^t(B_R)}+ \Vert g \Vert_{L^r (B_{R})} \right)^\sigma 
\end{align*}
hold for every pair of concentric balls $B_{R/4} \subset B_{R} \Subset \Omega$, where $c = c(n,p_i,\lambda, \Lambda,R, \Vert u \Vert_{\infty})$ and $\sigma= \sigma (n,p_i, r)$ are positive constants.
\end{thm}

\begin{rmk}
We observe that condition \eqref{Ipotesip} in Theorem \ref{thmBfinito} ensures that $u \in L^\infty_{\mathrm{loc}}(\Omega)$. Indeed, by Lemma \ref{LemmaTroisi}, the assumption $\omega \in  W_{\mathrm{loc}}^{1,\frac{\mathbf{p}+2}{\mathbf{p}+1}}(\Omega)$ implies $\omega \in L^t_{\mathrm{loc}}(\Omega)$, where $t$ was defined at \eqref{D1}. Thus, assuming that the exponents $t$ and $p_i$ satisfy the relation $p_n < \frac{\overline{p}^*}{t'}$, we obtain that local minimizers of \eqref{functional} are locally bounded in $\Omega$ (see Theorem \ref{ulimitatoCupini} below).
If $\omega \in L^\infty_{\mathrm{loc}}(\Omega)$, then this relation becomes $p_n  < \overline{p}^*$, which is the usual one when dealing with bounded minimizers (see \cite{CMMjota}). 
If $\overline{p} \ge n$, this condition disappears, since according to \eqref{definizionePi} we can choose $\overline{p}^*$ large enough so that $p_1+2 < \frac{\overline{p}^*}{t'}$. \\
Moreover, since $p_i \ge 2$, for every $i=1,...,n$, it holds $\overline{q} < \overline{p}$. Therefore, 
 if $\overline{p}<n$,  the exponent $t$ appearing in \eqref{Ipotesip} is equal to $\frac{n \overline{q}}{n-\overline{q}}$.\\
 Eventually, we note that the exponents $\overline{p}^*$ and $t$ depend on $p_i$, hence the condition $p_n < \frac{\overline{p}^*}{t'}$ is necessary, since is not satisfied for any choice of exponents $p_i$.
\end{rmk}

Now, we briefly describe the strategy of the proof, which, as usual, will consist of an approximation argument, a uniform a priori estimate, and a passage to the limit. It is well known that the a priori estimate needs to be established for sufficiently regular minimizers. Therefore, the first step is to establish higher differentiability for minimizers of functionals with Lipschitz continuous coefficients, since this result is not available in the literature. As usual the higher differentiability is achieved with the use of the well known difference quotient methods that is frequently used in this context.\\
More specifically, for a given $\varepsilon \geq 0$, we investigate the function 
\[
f_\varepsilon(x,u, \xi) = \sum_{i=1}^{n} \dfrac{b_i(x)}{p_i} |\xi_i|^{p_i} + \varepsilon \left( 1 + |\xi|^{\frac{p_n}{2}} \right)^2-\omega(x)u
\]
where the coefficients $b_i(x): \Omega \to \mathbb{R}$ are non-negative and satisfy the following condition
\[
L_i = \sup_{x, y \in \Omega, x \neq y} \frac{|b_i(x) - b_i(y)|}{|x - y|} < \infty.
\]
The first goal is to analyze the regularity properties of minimizers of the following functional
\[
\mathcal{F}_\varepsilon(u, \Omega) = \int_\Omega f_\varepsilon(x,u, Du) \, dx.
\]
After, we aim to establish an uniform a priori estimate, which plays a crucial role in the proof of our main result. Finally, Theorem \ref{thmBfinito} follows by proving that the a priori estimate is preserved in passing to the limit.\\

The plan of the paper is briefly described. After summarizing some known results and fixing a few notations, we prove a lemma for approximating functionals to be used in the a priori estimates in Section \ref{PreREG}. Section \ref{Apriorisec} is devoted to the proof of some a priori estimates for solutions to a family of approximating problems. Next, in Section \ref{Appro}, we pass to the limit in the approximating problems.
\section{Preliminaries}
In this section we introduce some notations and collect several results that we shall use to establish our main result.
We   denote by $c$ or $C$ a general constant that may vary on different occasions, even within the same line of estimates. 

In what follows, $B(x,r)=B_{r}(x)= \{ y \in \mathbb{R}^{n} : |y-x | < r  \}$ will denote the ball centered at $x$ of radius $r$. We will omit the indication of the center $x$ when no confusion arises.

Now, we recall a well-known iteration lemma (see \cite[Lemma $6.1$]{giusti}).
\begin{lem}\label{lm3} 
Let $Z(t)  :  [\rho,R] \rightarrow \mathbb{R}$ be a bounded nonnegative function, where $R>0$. Assume that for $\rho \leq t  < s \leq R$ it holds
$$Z(t) \leq \theta Z(s) +A(s-t)^{- \alpha} + B(s-t)^{- \beta}+ C$$
where $\theta \in (0,1)$, $A$, $B$, $C \geq 0, \alpha> \beta >0$  are constants. Then there exists a constant $c=c(\alpha, \theta, \beta)$ such that
$$Z(\rho) \leq c \biggl( A(R-\rho )^{- \alpha} + B(R-\rho )^{- \beta}+ C\biggr).$$
\end{lem}

At this point, we recall a result that will become significant later (for the proof see  \cite[Proposition $4.3$]{LemmaBrasco}).

\begin{prop}\label{LemmaBrasco}
 Let $u \in W^{1,\mathbf{p}}_{\mathrm{loc}}(\Omega)$ and assume that there exists   $k \in \{1,2,\dots,n \}$ such that
    $$|u_{x_k x_k}|^2|u_{x_k}|^{p_k-2} \in L^{1}_{\mathrm{loc}}(\Omega).$$
   If $u \in L^{\infty}_{\mathrm{loc}}(\Omega),$  then $u_{x_k} \in L_{\mathrm{\mathrm{loc}}}^{p_k+2}(\Omega)$  and there exist constants $C=C(n, p_k )>0$ and $\gamma= \gamma(p_k)$ such that for every pair of concentric balls $B_{\rho} \subset B_{R}\Subset  \Omega$,  it holds the following
    \begin{equation}
        \int_{B_\rho} |u_{x_k}|^{p_k+2}  dx \leq C ||u||^2_{\infty} \left(  \int_{B_R} |u_{x_k x_k}|^{2}|u_{x_k}|^{p_k-2}  dx+ \left(\frac{1}{R- \rho} \right)^\gamma \int_{B_R} |u_{x_k}|^{p_k} dx\right).
    \end{equation}
\end{prop}

 We shall use the following local boundedness result for minimizers of the anisotropic functional
defined at \eqref{functional}, whose proof can be found in \cite[Corollary 2.4]{uLimitato}.
 \begin{thm}\label{ulimitatoCupini}
    Let $f = f(x, \xi)$ be the integrand defined at \eqref{integrand} and assume that $\omega\in L^t(\Omega)$, with $\left(\frac{\overline{p}^*}{p_n} \right)' < t \le \infty$,
where $\bar{p}^*$ is the Sobolev exponent of $\bar{p}$, with $\bar{p}$ defined at \eqref{definizionePi}. 
Then, every local minimizer $u$ of \eqref{functional} is locally bounded in $\Omega$.
\end{thm}

\subsection{Algebraic inequalities}
For the function $V_{p}(\xi)$, defined at \eqref{DefVp}, we recall the following estimate (see e.g. \cite[Lemma 8.3]{giusti}). 
\begin{lem}\label{D1}
Let $1<p<+\infty$. There exist positive constants $c_1$ and $c_2$, depending only on $n$ and $p$, such that
\begin{center}
{$c_1(|\xi|^{2}+|\eta|^{2})^{\frac{p-2}{2}} \leq \dfrac{|V_{p}(\xi)-V_{p}(\eta)|^{2}}{|\xi-\eta|^{2}} \leq c_2(|\xi|^{2}+|\eta|^{2})^{\frac{p-2}{2}} $},
\end{center}
for any $\xi, \eta \in \mathbb{R}^{n}$, with $\xi \neq \eta$.
\end{lem}
For further needs we state the following lemma for the function ${Z_\delta (\xi)= (1+|\xi|^{2})^\frac{\delta}{2} \xi} $, whose proof  is obtained by combining \cite[formula 2.1]{Gia} and \cite[Lemma 2.2]{Fusco}.
\begin{lem}\label{Fusco}
Let $\delta \geq 0$. There exists a constant $c_1=c_1(\delta)>0$  such that
  $$ c_1 \langle {Z}_\delta (\xi)- {Z}_\delta (\eta), \xi- \eta \rangle \geq   \, |\xi-\eta|^{2}(1+|\xi|^{2}+|\eta|^{2})^{\frac{\delta}{2}},$$
  for any $\xi, \eta \in \mathbb{R}^{n}$.
\end{lem}

We conclude this section with an algebraic inequality, which can be found in \cite{Duzaar, giusti}.
\begin{lem}\label{Duzaar}
  For any $\alpha > 0$, there exists a constant $c = c(\alpha)$ such that, for all $\eta, \zeta \in \mathbb{R}^n \setminus \{0\}$, we have
\[
\frac{1}{c} \left||\eta|^{\alpha - 1} \eta - |\zeta|^{\alpha - 1} \zeta  \right| \leq \left( |\eta| + |\zeta| \right)^{\alpha - 1} |\eta - \zeta| \leq c \left| \, |\eta|^{\alpha - 1} \eta - |\zeta|^{\alpha - 1} \zeta \,  \right|.
\]
\end{lem}

\subsection{Difference quotient}
\label{secquo}
In this section we recall some properties of the finite difference quotient operator that will be needed in the sequel. 

\begin{dfn}
Let $F$ be a function defined in an open set $\Omega \subset \mathbb{R}^n$, let $h$ be a real number,  the finite difference operator $\tau_{s,h}F(x)$ is defined as follows
$$ 
\tau_{s,h}F(x) :=F(x+he_s)-F(x) ,$$
where $e_s$ denotes the direction of the $x_s$ axis.
\end{dfn}
The function $\tau_{s,h}F$ is defined in the set
$$\Omega_{|h|}: = \{ x \in \Omega : \mathrm{dist}(x,\partial \Omega)> |h|  \}= \{  x \in \Omega : x+he_s \in \Omega \}.$$
We start with the description of some elementary properties that can be found, for example, in \cite{giusti}. When no confusion can arise, we shall omit the index $s$, and we shall write simply $\tau_{h}$ instead of $ \tau_{s, h}$
\begin{prop}\label{rapportoincrementale}
Let $F \in W^{1,p}(\Omega)$, with $p \geq1$, and let $G:\Omega \rightarrow \mathbb{R}$ be a measurable function.
Then:
\\(i) $\tau_{h}F \in W^{1,p}(\Omega_{|h|})$ and 
$$D_{i}(\tau_{h}F)=\tau_{h}(D_{i}F).$$
(ii) If at least one of the functions $F$ or $G$ has support contained in $\Omega_{|h|}$, then
$$\displaystyle\int_{\Omega}F \tau_h G   dx = \displaystyle\int_{\Omega} G \tau_{-h}F dx.$$
(iii) We have $$\tau_{h} (FG)(x)= F(x+h)\tau_{h} G(x)+G(x) \tau_{h} F(x).$$
\end{prop}
The next result about the finite difference operator is a kind of integral version of Lagrange Theorem.
\begin{lem}\label{ldiff}
If $0<\rho<R,$ $|h|<\frac{R-\rho}{2},$ $1<p<+\infty$ and $F, \, D_s F\in L^{p}(B_{R})$, then
\begin{center}
$\displaystyle\int_{B_{\rho}} |\tau_{s, h}F(x)|^{p} dx \leq c(n,p)|h|^{p} \displaystyle\int_{B_{R}} |D_sF(x)|^{p} dx$.
\end{center}
Moreover,
\begin{center}
$\displaystyle\int_{B_{\rho}} |F(x+h)|^{p} d x \leq  \displaystyle\int_{B_{R}} |F(x)|^{p}d x$.
\end{center}
\end{lem}

We conclude by recalling the following
\begin{lem}\label{Lemmahzero}
If $0<\rho<R,$ $|h|<\frac{R-\rho}{2},$ $1<p<+\infty$ and $F\in L^{p}(B_{R})$. If there exists a positive constant $C$ such that
\begin{align}
\displaystyle\int_{B_{\rho}} |\tau_{s, h}F(x)|^{p} dx \leq C|h|^{p}, 
\end{align}
for every $h$, then the distributional derivative $D_s F$ belongs to $L^p(B_\rho).$\\
Moreover it holds
\begin{align}
\displaystyle\int_{B_{\rho}} |DF(x)|^{p} dx \leq C .
\end{align}
\end{lem}

\subsection{Anisotropic Sobolev spaces}
Here, we collect some well-known inequalities in the setting of anisotropic Sobolev spaces.
\\Let $p_i \ge 1$, $i=1,...,n$. We denote by $\overline{p}$ the harmonic average of $p_i$ and by $\overline{p}^*$ the Sobolev conjugate exponent of $\overline{p}$ defined at \eqref{definizionePi}.
    
    We recall the following embedding results for anisotropic Sobolev spaces. We refer to \cite{ Acerbi, feopass,troisi}.
 \begin{lem}\label{anisopoincare}
      Let $E \subset \mathbb{R}^n$ be a bounded open set and let $  p_i \geq 1$, for all $i=1, ..., n$. If $u \in W_0^{1, \mathbf{p}}(E)$, then there exists a positive constant $c $ depending on $n$, $p_i$ and, only in the case $\overline{p} \ge n$, also on $\overline{p}^*$ and $E$, such that
      \begin{align*}
      \lvert \lvert u \rvert \rvert_{L^{\overline{p}^*}(E)}\leq c \sum_{i=1}^{n} \lvert \lvert u_{x_i} \rvert \rvert_{L^{p_i}(E)}.
   \end{align*}
 \end{lem}   
\begin{lem}\label{LemmaTroisi}
    Let $E \subset \mathbb{R}^n$ be a bounded open set and consider $u \in W^{1, \mathbf{p}}(E)$, with $  p_i \geq 1$, for all $i=1, ..., n$.\\
    If $\overline{p} \le n$, then there exists a positive constant $\gamma_1$ depending on $n, p_i$ and, only in the case $\overline{p}=n$, also on $\overline{p}^*$ and $E$, such that
    \begin{align*}
        \lvert \lvert u \rvert \rvert_{L^{\overline{p}^*}(E)} \leq \gamma_1  \left[ \sum_{i=1}^{n} \lvert \lvert u_{x_i} \rvert \rvert_{L^{p_i}(E)} + \lvert \lvert u \rvert \rvert_{L^1(E)} \right].
    \end{align*}
      If $\overline{p} > n$,  there exists a positive constant $\gamma_2$ depending only on $ n, p_i$ and $E$, such that
    \begin{align*}
        \lvert \lvert u \rvert \rvert_{L^{\infty}(E)} \leq \gamma_2  \left[ \sum_{i=1}^{n} \lvert \lvert u_{x_i} \rvert \rvert_{L^{p_i}(E)} + \lvert \lvert u \rvert \rvert_{L^1(E)} \right].
    \end{align*}
\end{lem}

\section{A regularized problem}\label{PreREG}
In this section we prove a regularity result for more regular problems with respect to \eqref{functional} that will be needed in the approximation procedure. More precisely, for $\varepsilon \geq 0$ we consider 
\begin{equation}\label{Eqconepsilon}
 f_\varepsilon(x,u, \xi)= \sum_{i=1}^{n} \, \dfrac{b_i(x)}{p_i} \lvert \xi_i \rvert^{p_i} + \varepsilon(1+|\xi|^\frac{p_n}{2})^2  -\overline{\omega}(x)u
\end{equation}
with
$b_i(x):\Omega \to [0,+\infty), \, i=1, \dots, n$, non-negative, Lipschitz continuous coefficients,
i.e., the following conditions hold
\begin{equation}\label{atilde}
L_i= \sup_{x,y \in \Omega, x \neq y} \frac{|b_i(x)-b_i(y)|}{|x-y|}< \infty , \quad i=1,..,n.
\end{equation}
Setting $\tilde{f}(x, \xi)= \sum_{i=1}^{n} \, b_i(x) \lvert \xi_i \rvert^{p_i},$ one can easily check that
\begin{equation}\label{A'4}
    |D_\xi  \tilde{f}(x,\xi)-D_\xi \tilde{f}(y, \xi)| \leq K |x-y|\, \, \sum_{i=1}^{n} |\xi_i|^{p_i-1} \tag{A3'}
\end{equation}
\noindent for a.e.\ $x,y \in \Omega $ and every $\xi \in \mathbb{R}^{ n}$, where $K=K(L_i, p_i)$
.\\
We recall a Lipschitz regularity Theorem for the minimizers of the functional 
\begin{equation}\label{functional2}
    \mathcal{F}_\varepsilon (x, Dv) =\int_\Omega f_\varepsilon(x,v, Dv) \, dx,
\end{equation}
that can be deduced by  \cite[Theorem $8.9$]{giusti} in the case $p=p_n$, where $f_\varepsilon$ was defined at \eqref{Eqconepsilon}. 
\begin{thm}\label{thm2Mascolo}
    Let $\tilde{f} $ satisfy \eqref{A2}, \eqref{A3} and \eqref{A'4} and $\overline{\omega} \in L^\infty(\Omega)$.
    Let $u_\varepsilon \in W^{1,p_n}_{\mathrm{loc}}(\Omega)$ be a local minimizer of \eqref{functional2}.
Then, $u_\varepsilon \in W^{1,\infty}_{\mathrm{loc}}(\Omega)$ 
\end{thm}

We shall use the higher differentiability of the minimizers $u_\varepsilon$ of $\mathcal{F}_\varepsilon$ that, as far  we know, is not available in literature and that is contained in the following
{
\begin{lem}\label{LemmaInduzione}
  Let $\tilde{f}$ satisfy \eqref{A2}, \eqref{A3} and \eqref{A'4} for exponents $ p_i \geq 2,  \forall i=1,\dots,n$, and assume {$\overline{\omega} \in L^\infty(\Omega) \cap W^{1,\frac{\mathbf{p}+2}{\mathbf{p}+1}}_{\mathrm{loc}}(\Omega)$. }
    Let  $u_\varepsilon \in W^{1,p_n}_{\mathrm{loc}}(\Omega) $ be a local minimizer of the functional \eqref{functional2}.
    Then, 
    we have  $${V_{p_i}((u_\varepsilon)_{x_i})} \in W^{1,2}_{\mathrm{loc}}(\Omega), \quad \forall i =1,\dots, n.   $$
   
\end{lem}
\begin{proof}
We start by observing that $u_\varepsilon$
 solves the following Euler-Lagrange system
{\begin{equation}\label{System}
    \int_{\Omega} \langle f_{\xi} (x, D u_\varepsilon(x))+ \varepsilon(1+ |D u_\varepsilon(x)|^2)^{\frac{p_n-2}{2}}D u_\varepsilon(x), D \varphi(x) \rangle \, dx = \int_{\Omega}\overline{\omega}(x)\varphi(x) \, dx , 
\end{equation}}
for every $ \varphi \in W^{1,p_n}_{0}(\Omega)$ and that by Theorem \ref{thm2Mascolo} $u_\varepsilon \in W_{\mathrm{loc}}^{1,\infty}(\Omega)$.

Fix a ball $B_R \Subset \Omega$ and consider radii $0<\rho <t< R$, a cut-off function $\eta \in C_0^{\infty}(B_t)$, with $\eta=1$ on $B_{\rho}$, $0 \leq \eta \leq 1$, $|D \eta | \leq \frac{c}{t-\rho}$  and $|h|\leq \frac{R - t}{2}$. We test \eqref{System} with the function
$$\varphi = \tau_{j, -h} \left(  \eta^2 \tau_{j,h} u_\varepsilon \right) $$
thus obtaining
\begin{align*}
    &\int_{\Omega} \langle f_{\xi} (x, D u_\varepsilon(x))+ \varepsilon(1+ |D u_\varepsilon(x)|^2)^{\frac{p_n-2}{2}}D u_\varepsilon(x), \tau_{j,-h} D  \left(  \eta^2 \tau_{j,h} u_\varepsilon(x) \right) \rangle \, dx \\
    &=  \int_{\Omega} \overline{\omega}(x)\tau_{j,-h}\left(\eta^2 \tau_{j,h} u_\varepsilon(x) \right) \, dx \, ,
\end{align*}
and hence, by $(ii)$ of Proposition \ref{rapportoincrementale},
\begin{align}\label{IntEq}
    &\int_{\Omega} \langle \tau_{j,h} f_{\xi} (x, D u_\varepsilon(x))+ \varepsilon\tau_{j,h}(1+ |D u_\varepsilon(x)|^2)^{\frac{p_n-2}{2}}D u_\varepsilon(x), D  \left(  \eta^2 \tau_{j,h} u_\varepsilon(x) \right) \rangle \, dx \notag\\
    &= {\int_{\Omega} \tau_{j,h}\overline{\omega}(x)\eta^2 \tau_{j,h} u_\varepsilon(x)  \, dx \,  .}
\end{align}
Since
$$D(\eta^2 \tau_{j,h} (u_\varepsilon)) = 2 \eta D \eta \tau_{j,h} (u_\varepsilon) + \eta^2 \tau_{j,h} D u_\varepsilon$$
and
\begin{align*}
&\tau_{j,h} f_{\xi} (x, D u_\varepsilon(x))+ \varepsilon\tau_{j,h} {\left((1+ |D u_\varepsilon(x)|^2)^{\frac{p_n-2}{2}}D u_\varepsilon(x) \right)} \\
&=   f_\xi (x+e_jh, D u_\varepsilon(x+e_jh))- f_\xi (x+e_jh, D u_\varepsilon(x)) \\
& \qquad + f_\xi (x+e_jh, D u_\varepsilon(x))- f_\xi (x, D u_\varepsilon(x))\\
&  \qquad + \varepsilon {\tau_{j, h} \left((1+ |D u_\varepsilon(x)|^2)^\frac{p_n-2}{2}D u_\varepsilon(x) \right)},
\end{align*}
we  may rewrite equality \eqref{IntEq} as follows
\begin{align}\label{s}
 0=\int_{\Omega} &\langle f_{\xi} (x+e_jh, D u_\varepsilon(x+e_jh))- f_{\xi} (x+e_jh, D u_\varepsilon(x)) , \eta^2  \tau_{j,h} D u_\varepsilon(x) \rangle \, dx \notag \\
    &+ 2 \int_{\Omega} \langle f_{\xi} (x+e_jh, D u_\varepsilon(x+e_jh))- f_{\xi} (x+e_jh, D u_\varepsilon(x)) ,  \eta D \eta \tau_{j,h} u_\varepsilon(x) \rangle \, dx \notag \\
    &+ 2 \int_{\Omega} \langle f_{\xi} (x+e_jh, D u_\varepsilon(x))- f_{\xi} (x, D u_\varepsilon(x)) ,  \eta D \eta \tau_{j,h} u_\varepsilon(x) \rangle \, dx \notag \\
    &+  \int_{\Omega} \langle f_{\xi} (x+e_jh, D u_\varepsilon(x))- f_{\xi} (x, D u_\varepsilon(x)) , \eta^2  \tau_{j,h} D u_\varepsilon(x) \rangle \, dx \notag \\
    &+ 2 \varepsilon  \int_{\Omega} \langle {\tau_{j, h} \left((1+ |D u_\varepsilon(x)|^2)^\frac{p_n-2}{2}D u_\varepsilon(x) \right)},  \eta D \eta \tau_{j,h} u_\varepsilon(x) \rangle \, dx  \notag \\
     &+ \varepsilon  \int_{\Omega} \langle {\tau_{j, h} \left((1+ |D u_\varepsilon(x)|^2)^\frac{p_n-2}{2}D u_\varepsilon(x) \right)},  \eta^2  \tau_{j,h} D u_\varepsilon(x) \rangle \, dx \notag \\
&-
     {\int_{\Omega} \tau_{j,h}\overline{\omega}(x)\eta^2 \tau_{j,h} u_\varepsilon(x)  \, dx} \notag\\
    & =: J_1 + J_2 + J_3 + J_4 + J_5+ J_6+J_7. \notag
\end{align}
Therefore
\begin{equation}\label{SommaJ}
   J_1 +J_6 \leq |J_2|  + |J_3| + |J_4| +|J_5|+|J_7|  .
\end{equation}
 From \eqref{A2}, we infer
\begin{align}\label{J_1}
    J_1 &=  \int_{\Omega} \langle f_{\xi} (x+e_j h, D u_\varepsilon(x+ e_jh))- f_{\xi} (x+ e_j h, D u_\varepsilon(x)) , \eta^2  \tau_{j,h} D u_\varepsilon (x)\rangle \, dx \notag \\
    & \geq l \int_{\Omega} \eta^2 \sum_{i=1}^{n}\, \left( \, |(u_\varepsilon)_{x_i}(x+e_jh)|^2+|(u_\varepsilon)_{x_i}(x) |^2 \, \right)^\frac{p_i -2}{2}| \tau_{j,h} (u_\varepsilon)_{x_i}(x) |^2 dx =: l \textbf{\textbf{RHS}}
\end{align}
The integral $J_6$ can be estimated by using Lemma \ref{Fusco} with $\delta= p_n-2$, as follows
\begin{equation}\label{J6}
    J_6 \geq \frac{\varepsilon}{c} \int_{\Omega} \eta^2 (1+ |D u_\varepsilon(x+ he_j)|^2 + |D u_\varepsilon(x)|^2)^\frac{p_n-2}{2} |\tau_{j, h} D u_\varepsilon(x)|^2 =: \frac{\varepsilon}{c} \overline{\textbf{\textbf{RHS}}},
\end{equation}
where $c=c(p_n)$ is the constant appearing in Lemma \ref{Fusco}.\\
Now we consider the term $|J_2|$. By hypotesis \eqref{A3} and by  Young's  inequality, we get
\begin{align}\label{J_2}
    |J_2| &= 2\left|   \int_{\Omega} \langle f_{\xi} (x+e_jh, D u_\varepsilon(x+e_jh))- f_{\xi} (x+e_jh, D u_\varepsilon(x)) ,  \eta \,   \eta_{x_i} \tau_{j,h} u_\varepsilon \rangle \, dx  \right| \notag \\
    & \leq 2 c \int_{\Omega} \eta \sum_{i=1}^{n}\, \left( \, |(u_\varepsilon)_{x_i}(x+e_jh)|^2+|(u_\varepsilon)_{x_i}(x) |^2 \, \right)^\frac{p_i -2}{2}| \tau_{j,h} (u_\varepsilon)_{x_i}| | \tau_{j,h} u_\varepsilon ||\eta_{x_i}| dx  \notag \\
& \leq \beta \underbrace{\int_{\Omega}  \eta^2 \sum_{i=1}^{n}\, \left( \, |(u_\varepsilon)_{x_i}(x)(x+e_jh)|^2+|(u_\varepsilon)_{x_i}(x) |^2 \, \right)^\frac{p_i -2}{2}| \tau_{j,h} (u_\varepsilon)_{x_i} |^2 dx}_{\textbf{\textbf{RHS}}}  \notag \\
& \qquad + c_\beta \int_{\Omega} \sum_{i=1}^{n}\, \left( \, |(u_\varepsilon)_{x_i}(x+e_jh)|^2+|(u_\varepsilon)_{x_i}(x) |^2 \, \right)^\frac{p_i -2}{2}  | \tau_{j,h} u_\varepsilon |^2 \, |\eta_{x_i}|^2 dx  \notag \\
& \leq  \beta \,\textbf{\textbf{RHS}} \,  + \frac{c_\beta(\Vert Du_\varepsilon \Vert_\infty)}{(t-\rho)^2}  \int_{B_t}   | \tau_{j,h} u_\varepsilon |^2 dx \notag \\
& \leq \beta \,\textbf{\textbf{RHS}} \,  + \frac{c_\beta(R, \Vert Du_\varepsilon \Vert_\infty)}{(t-\rho)^2} |h|^2\int_{B_{R}}   \, |(u_\varepsilon)_{x_j}(x)|^{2}dx \notag\\
& \leq  \beta \,{\textbf{RHS}} \,  + \frac{c_\beta(R, \Vert Du_\varepsilon \Vert_\infty)}{(t-\rho)^2} |h|^2,
\end{align}
where to estimate the last integral in the right-hand side  we used Lemma \ref{ldiff} and that $Du_\varepsilon$ is locally bounded.

Now, we turn our attention to $|J_3|$. From condition \eqref{A'4} and the properties of $\eta$, we get
\begin{align}\label{J_3}
    |J_3|&= \left| \int_{\Omega} \langle f_{\xi} (x+e_j h, D u_\varepsilon(x))- f_{\xi} (x, D u_\varepsilon(x)) , 2 \eta D \eta \tau_{j,h} u_\varepsilon \rangle \, dx  \right| \notag \\
    & \leq |h|\, K\,  \sum_{i=1}^n \int_{\Omega} 2\eta | \eta_{x_i}||\tau_{j, h} u_\varepsilon| \,|(u_\varepsilon)_{x_i}|^{p_i-1} dx \notag \\
   & \leq |h| \frac{c(K)}{t-\rho}\, \sum_{i=1}^n  \, \left(   \int_{B_t}  |\tau_{j,h}u_\varepsilon|\,|(u_\varepsilon)_{x_i}|^{p_i-1} dx \right) \notag \\
   & \leq |h|^2 \frac{c(K, R)}{t-\rho}\,  \sum_{i=1}^{n} \lvert \lvert  (u_\varepsilon)_{x_i} \rvert \rvert_{L^\infty(B_R)}^{p_i-1} \left(\int_{B_{R}}   \, |(u_\varepsilon)_{x_j}(x)|^{p_j}dx \right)^{\frac{1}{p_j}} \notag\\
   & \leq |h|^2 \frac{c(K, R, \Vert Du_\varepsilon \Vert_\infty)}{t-\rho},
\end{align}
where, as before, we used Theorem \ref{thm2Mascolo}, H\"older's inequality and  Lemma \ref{ldiff}.\\
Similarly as for the estimate of $J_3$, from condition \eqref{A'4} and the properties of $\eta$, we obtain
\begin{align}\label{J_4}
     |J_4| &= \left|  \int_{\Omega} \langle f_{\xi} (x+e_jh, D u_\varepsilon(x))- f_{\xi} (x, D u_\varepsilon(x)) , \eta^2  \tau_{j,h} D u_\varepsilon \rangle \, dx \right| \notag \\
     &= \left|  \int_{\Omega} \sum_{i=1}^{n} \eta^2 ( b_i(x+e_jh)- b_i(x) )|(u_\varepsilon)_{x_i}|^{p_i -1} \, \tau_{j, h}(u_\varepsilon)_{x_i} dx\right| \notag \\
    & \leq c |h|\, K\, \sum_{i=1}^{n}  \int_{\Omega} \eta^2 |\tau_{j, h}(u_\varepsilon)_{x_i}| \, |(u_\varepsilon)_{x_i}|^{p_i -1} dx \notag \\
    & \leq \beta \underbrace{\int_{\Omega} \eta^2 \sum_{i=1}^{n}\, \left( \, |(u_\varepsilon)_{x_i}(x+e_j h)|^2+|(u_\varepsilon)_{x_i}(x) |^2 \, \right)^\frac{p_i -2}{2}| \tau_{j,h} (u_\varepsilon)_{x_i} |^2 dx}_{\textbf{\textbf{RHS}}}  \notag \\
    & \qquad + c_{\beta} |h|^2 \, \sum_{i=1}^{n}  \int_{B_t} \eta^2 |(u_\varepsilon)_{x_i}|^{p_i} \,  dx \notag \\
     & \leq  \beta \textbf{\textbf{RHS}} + c_{\beta} |h|^{2} \,\sum_{i=1}^{n}   \int_{B_t}  |(u_\varepsilon)_{x_i}|^{p_i} \,  dx  \notag\\
     & \leq  \beta \textbf{\textbf{RHS}} + c_{\beta}(R, \Vert Du_\varepsilon \Vert_\infty) |h|^{2}.
\end{align}
Now, we take care of $|J_5|$. From Lemma \ref{D1}, Young's  inequality with exponents $(2,2)$  and the properties of $\eta$, we get
\begin{align}\label{J5}
|J_5| &\leq 2 \varepsilon  \int_{\Omega} |\tau_{j, h}((1+ |D u_\varepsilon(x)|^2)^\frac{p_n-2}{2} D u_\varepsilon(x))| \,  |\eta| |D \eta| |\tau_{j,h} u_\varepsilon|  \, dx \notag \\ 
&\leq c \varepsilon \int_{\Omega} (1+ |D u_\varepsilon(x+ he_j)|^2 + |D u_\varepsilon(x)|^2)^\frac{p_n-2}{2} |\tau_{j, h} D u_\varepsilon(x)| \,  \eta |D \eta| |\tau_{j,h} u_\varepsilon|  \, dx \notag \\
&\leq \frac{\varepsilon}{2c} \underbrace{   \int_{\Omega} \eta^2 (1+ |D u_\varepsilon(x+ he_j)|^2 + |D u_\varepsilon(x)|^2)^\frac{p_n-2}{2} |\tau_{j, h} D u_\varepsilon(x)|^2   \, dx }_{ \overline{\textbf{\textbf{RHS}}}}\notag \\
& \qquad + \frac{c \varepsilon}{(t-\rho)^2} \int_{B_t} (1+ |D u_\varepsilon(x+ he_j)|^2 + |D u_\varepsilon(x)|^2)^\frac{p_n-2}{2}   |\tau_{j,h} u_\varepsilon|^2  \, dx \notag \\ 
& \leq \frac{\varepsilon}{2c}\overline{\textbf{\textbf{RHS}}}\ +  \frac{c(\Vert Du_\varepsilon \Vert_\infty) \varepsilon}{(t-\rho)^2}   \int_{B_t}  |\tau_{j,h} u_\varepsilon|^{2} dx  \, \notag \\
& \leq  \frac{\varepsilon}{2c}\overline{\textbf{\textbf{RHS}}}\ +  \frac{c(\Vert Du_\varepsilon \Vert_\infty) \varepsilon}{(t-\rho)^2} |h|^2  \int_{B_R}  |D u_\varepsilon|^{2} dx  \notag\\
& \leq \frac{\varepsilon}{2c}\overline{\textbf{\textbf{RHS}}}\ +  \frac{c(R,\Vert Du_\varepsilon \Vert_\infty) \varepsilon}{(t-\rho)^2} |h|^2   ,
\end{align}
where we used that $Du_\varepsilon$ is locally bounded, by Theorem \ref{thm2Mascolo}, and  Lemma \ref{ldiff}.\\
From the assumption on $\omega$ and since $u_\varepsilon$ is locally Lipschitz continuous in $\Omega$, we may use H\"older's inequality and Lemma \ref{ldiff}, thus getting the following estimate for $|J_7|$
\begin{align}
    |J_7| \le & \int_{B_t}|\tau_{j,h}\omega| \, |\tau_{j,h}u_\varepsilon| \, dx \notag\\
    \le &  \left( \int_{B_t}|\tau_{j,h}\omega|^{\frac{p_j+2}{p_j+1}} \, dx\right)^\frac{p_j+1}{p_j+2} \left( \int_{B_t}|\tau_{j,h}u_\varepsilon|^{p_j+2} \, dx\right)^\frac{1}{p_j+2} \notag\\
    \le & \, c |h|^2 \left( \int_{B_R}|\omega_{x_j}|^{\frac{p_j+2}{p_j+1}} \, dx\right)^\frac{p_j+1}{p_j+2} \left( \int_{B_R}|(u_\varepsilon)_{x_j}|^{p_j+2} \, dx\right)^\frac{1}{p_j+2} \notag\\
    \le & \,  c |h|^2 \sum_{i=1}^n \left( \int_{B_R}|\omega_{x_i}|^{\frac{p_i+2}{p_i+1}} \, dx\right)^\frac{p_i+1}{p_i+2}
    , \label{j_7}
\end{align}
where we also used the properties of $\eta$.

Inserting \eqref{J_1}, \eqref{J6}, \eqref{J_2}, \eqref{J_3}, \eqref{J_4}, \eqref{J5} and \eqref{j_7}  in \eqref{SommaJ} 
and reabsorbing the term $\overline{\textbf{RHS}}$
, we obtain
\begin{align}\label{sum}
 l &\int_{\Omega} \eta^2 \sum_{i=1}^{n}\, \left( \, |(u_\varepsilon)_{x_i}(x+e_jh)|^2+|(u_\varepsilon)_{x_i}(x) |^2 \, \right)^\frac{p_i -2}{2}| \tau_{j,h} (u_\varepsilon)_{x_i} |^2 dx + \frac{\varepsilon}{2c} \overline{\textbf{\textbf{RHS}}} \notag \\
&  \leq  2 \beta \, \int_{\Omega}  \eta^2 \sum_{i=1}^{n}\, \left( \, |(u_\varepsilon)_{x_i}(x+e_jh)|^2+|(u_\varepsilon)_{x_i}(x) |^2 \, \right)^\frac{p_i -2}{2}| \tau_{j,h} (u_\varepsilon)_{x_i} |^2 dx \notag\\
& \qquad +   \frac{c_\beta}{(t-\rho)^2} |h|^2 +   
\frac{c(K)}{(t-\rho)}|h|^2 
+ c_{\beta} |h|^{2} 
 +  \frac{c \varepsilon}{(t-\rho)^2} |h|^2  \notag\\
& \qquad + c|h|^2 \sum_{i=1}^n \left( \int_{B_R}|\omega_{x_i}|^{\frac{p_i+2}{p_i+1}} \, dx\right)^\frac{p_i+1}{p_i+2}. 
\end{align}
Since $\overline{\textbf{RHS}}$ is non-negative, we obviously have
\begin{align}\label{RHS2}
      \frac{l}{2} &\int_{\Omega} \eta^2 \sum_{i=1}^{n}\, \left( \, |(u_\varepsilon)_{x_i}(x+e_jh)|^2+|(u_\varepsilon)_{x_i}(x) |^2 \, \right)^\frac{p_i -2}{2}| \tau_{j,h} (u_\varepsilon)_{x_i} |^2 dx \notag \\
 & \leq  \frac{l}{2}\int_{\Omega} \eta^2 \sum_{i=1}^{n}\, \left( \, |(u_\varepsilon)_{x_i}(x+e_jh)|^2+|(u_\varepsilon)_{x_i}(x) |^2 \, \right)^\frac{p_i -2}{2}| \tau_{j,h} (u_\varepsilon)_{x_i} |^2 dx + \frac{\varepsilon}{2c} \overline{\textbf{\textbf{RHS}}}.
\end{align}
Hence, from \eqref{sum} and \eqref{RHS2}, choosing $\beta = \frac{l}{4}$ and reabsorbing the first  term in the right-hand side  by the left-hand side, we get
\begin{align}\label{stimaUnioneJ}
   \frac{l}{2} &\int_{\Omega} \eta^2 \sum_{i=1}^{n}\, \left( \, |(u_\varepsilon)_{x_i}(x+e_jh)|^2+|(u_\varepsilon)_{x_i}(x) |^2 \, \right)^\frac{p_i -2}{2}| \tau_{j,h} (u_\varepsilon)_{x_i} |^2 dx \notag \\
& \leq   \frac{c_\beta}{(t-\rho)^2} |h|^2 +   
\frac{c(K)}{(t-\rho)}|h|^2 
+ c_{\beta} |h|^{2} 
 +  \frac{c \varepsilon}{(t-\rho)^2} |h|^2  \notag\\
& \qquad + c|h|^2 \sum_{i=1}^n \left( \int_{B_R}|\omega_{x_i}|^{\frac{p_i+2}{p_i+1}} \, dx\right)^\frac{p_i+1}{p_i+2} ,
\end{align}
with a constant $c=c(l,K, R)$.\\
Applying Lemma \ref{D1} in the left-hand side of \eqref{stimaUnioneJ} and using that $\eta=1$ on $B_\rho$, we get 

\begin{align}
 \sum_{i=1}^n \int_{B_{\rho}} |\tau_{j,h} V_{p_i} ((u_\varepsilon)_{x_i})|^2 dx & \leq \,|h|^2 \Biggl\{   \frac{c_\beta}{(t-\rho)^2}  +   
\frac{c(K)}{(t-\rho)} 
+ c_{\beta}  
 +  \frac{c \varepsilon}{(t-\rho)^2}   \notag\\
& \qquad +  \sum_{i=1}^n \left( \int_{B_R}|\omega_{x_i}|^{\frac{p_i+2}{p_i+1}} \, dx\right)^\frac{p_i+1}{p_i+2} \Biggr\}, \notag 
  \end{align}
 and so
\begin{align}
  &\sum_{i=1}^n \int_{B_{\rho}} |\tau_{j,h} V_{p_i} ((u_\varepsilon)_{x_i})|^2 dx \leq  c |h|^{2} \,\left[ 1+ \sum_{i=1}^{n}   \Vert \omega_{x_i} \Vert_{L^{\frac{p_i+2}{p_i+1}}(B_{R})}  \right], \label{StimaLemma}
\end{align}
for a positive constant $c = c(n,p_i,\rho, R, \lambda, K, \lvert \lvert Du \rvert \rvert_{\infty})$.
Then we have that 
 $$V_{p_i}((u_\varepsilon)_{x_i}) \in W^{1,2}_{\mathrm{loc}}(\Omega), \qquad \forall i=1, \dots, n,$$ 
 which implies in particular that $$|(u_\varepsilon)_{x_i x_j}|^2|(u_\varepsilon)_{x_i}|^{p_i-2} \in L^1_{\mathrm{loc}}(\Omega), \qquad \forall i=1, \dots, n \text{ and  }  \forall j=1, \dots, n. $$
\end{proof}

\section{A priori estimates}\label{Apriorisec}
The main aim of this section is to establish the following a priori estimate which is the main step in the proof of our main result.
\begin{thm}\label{AppThm}
Let $\omega \in W^{1,\frac{\mathbf{p}+2}{\mathbf{p}+1}}_{\mathrm{loc}}(\Omega)$ and let $u \in  W^{1, \mathbf{p}}_{\mathrm{loc}}(\Omega) \cap L^\infty_{\mathrm{loc}}(\Omega)$  be a local minimizer of \eqref{functional} under assumptions \eqref{A2}--\eqref{A4}, with exponents $p_i\geq 2, \forall i=1,\dots,n $, such that $p_n < p_1+2$,
and with a function $g \in L^{r}_{\mathrm{loc}}{(\Omega)}$, where $r$ satisfies the condition
\eqref{Ipotesir}.
If we assume that 
\begin{equation}\label{Apriori}
V_{p_i}(u_{x_i}) \in {W^{1,2}_{\mathrm{loc}}(\Omega)}, \qquad\forall i=1,\dots,n,
\end{equation}
then the following estimates
\begin{align}\label{uxistima}
    \sum_{i=1}^{n} \left(  \int_{B_{R/4}} |u_{x_i}|^{p_i +2 } dx \right) & \leq c  \left(1+ \sum_{i=1}^{n} \left(\Vert u_{x_i} \Vert_{L^{p_i}(B_R)} + \Vert \omega_{x_i} \Vert_{L^\frac{p_i+2}{p_i+1}(B_{R})} \right) + \Vert g \Vert_{L^r (B_{R})} \right)^\gamma 
\end{align}
and
\begin{align}
        \sum_{i=1}^{n}  \left( \int_{B_{R/4}} |u_{x_i x_j}|^2|u_{x_i}|^{p_i-2}dx \right) & \leq c  \left( 1+\sum_{i=1}^{n} \left(\Vert u_{x_i} \Vert_{L^{p_i}(B_R)} + \Vert \omega_{x_i} \Vert_{L^\frac{p_i+2}{p_i+1}(B_{R})} \right)+ \Vert g \Vert_{L^r (B_{R})} \right)^\gamma 
 \label{StimaTeo1}
\end{align}
hold for every pair of concentric balls $B_{R/4} \subset B_{R} \Subset \Omega$, where $c = c(n,p_i,\lambda, \Lambda,R, {\Vert u \Vert_{\infty}})$ and $\gamma= \gamma (n,p_i, {r})$ are positive constants.
\end{thm}
\begin{proof}
We start by observing that, thanks to assumption \eqref{Ipotesip}, Theorem \ref{ulimitatoCupini} and Lemma \ref{LemmaTroisi}, we have $u \in L^\infty_{\mathrm{loc}}(\Omega)$. Therefore, by virtue of the local boundedness of $u$ and of assumption \eqref{Apriori}, an application of Proposition \ref{LemmaBrasco} gives that
\begin{equation}
    u_{x_k} \in L^{p_k+2}_{\mathrm{loc}}(\Omega) \qquad \text{for every } k \in \mathbb{N}. \label{highint}
\end{equation}
Fix a ball $B_R \Subset \Omega$ and consider radii $\frac{R}{4}< \rho< s<t<t'<{\sigma}<R$, a cut-off function $\eta \in C_0^{\infty}(B_t)$, with $\eta=1$ on $B_{s}$, $0 \leq \eta \leq 1$, $|D \eta | \leq \frac{c}{t-s}$  and $|h|\leq \frac{t' - t}{2}$.
We test the Euler-Lagrange equation of \eqref{functional} with the function $$\varphi = \tau_{j,-h}(\eta^2 \tau_{j,h}u),$$ 
thus getting
\begin{align}\label{SommaIntera}
 0= \int_{\Omega} &\langle f_{\xi} (x+e_jh, Du(x+e_jh))- f_{\xi} (x+e_jh, Du(x)) , \eta^2  \tau_{j,h} Du \rangle \, dx \notag \\
    &+ \int_{\Omega} \langle f_{\xi} (x+e_jh, Du(x+e_jh))- f_{\xi} (x+e_jh, Du(x)) , 2 \eta D \eta \tau_{j,h} u \rangle \, dx \notag \\
    &+  \int_{\Omega} \langle f_{\xi} (x+e_jh, Du(x))- f_{\xi} (x, Du(x)) , 2 \eta D \eta \tau_{j,h} u \rangle \, dx \notag \\
    &+  \int_{\Omega} \langle f_{\xi} (x+e_jh, Du(x))- f_{\xi} (x, Du(x)) , \eta^2  \tau_{j,h} Du \rangle \, dx \notag \\
    &-  \int_{\Omega} \tau_{j,h}\omega(x)\eta^2 \tau_{j,h} u(x)  \, dx \notag\\
    & =: I_1 + I_2 + I_3 + I_4+I_5. \notag
\end{align}
From the previous equality, we deduce that 
\begin{equation}\label{SommaInt}
   I_1 \leq |I_2|  + |I_3| + |I_4|+|I_5| .
\end{equation}
By virtue of assumption \eqref{A2}, we infer
\begin{align}\label{I1}
    I_1 &=  \int_{\Omega} \langle f_{\xi} (x+e_j h, Du(x+ e_jh))- f_{\xi} (x+ e_j h, Du(x)) , \eta^2  \tau_{j,h} Du \rangle \, dx \notag \\
    & \geq l \int_{\Omega} \eta^2 \sum_{i=1}^{n}\, \left( \, |u_{x_i}(x+e_jh)|^2+|u_{x_i}(x) |^2 \, \right)^\frac{p_i -2}{2}| \tau_{j,h} u_{x_i} |^2 dx =: l  \textbf{\textbf{RHS}}.
\end{align}

Now we consider the term $|I_2|$. By using hypotesis \eqref{A3}, Young's  inequality, we get
\begin{align}\label{termtau}
    |I_2| &= \left|   \int_{\Omega} \langle f_{\xi} (x+e_jh, Du(x+e_jh))- f_{\xi} (x+e_jh, Du(x)) , 2 \eta \,   \eta_{x_i} \tau_{j,h} u \rangle \, dx  \right| \notag \\
    & \leq 2 L \int_{\Omega} \eta \sum_{i=1}^{n}\, \left( \, |u_{x_i}(x+e_jh)|^2+|u_{x_i}(x) |^2 \, \right)^\frac{p_i -2}{2}| \tau_{j,h} u_{x_i} | | \tau_{j,h} u ||\eta_{x_i}| dx  \notag \\
& \leq 2 \varepsilon \underbrace{\int_{\Omega}  \eta^2 \sum_{i=1}^{n}\, \left( \, |u_{x_i}(x+e_jh)|^2+|u_{x_i}(x) |^2 \, \right)^\frac{p_i -2}{2}| \tau_{j,h} u_{x_i} |^2 dx}_{\textbf{RHS}}  \notag \\
& \qquad + c_\varepsilon \int_{\Omega} \sum_{i=1}^{n}\, \left( \, |u_{x_i}(x+e_jh)|^2+|u_{x_i}(x) |^2 \, \right)^\frac{p_i -2}{2}  | \tau_{j,h} u |^2 \, |\eta_{x_i}|^2 dx  \notag \\
& \leq 2 \varepsilon \,\textbf{\textbf{RHS}} \,  + \frac{c_\varepsilon}{(t-s)^2}  \int_{B_t} \sum_{i=1}^{n}\, \left( \, |u_{x_i}(x+e_jh)|^2+|u_{x_i}(x) |^2 \, \right)^\frac{p_i -2}{2}  | \tau_{j,h} u |^2 dx  \notag \\
& \leq 2 \varepsilon \,\textbf{\textbf{RHS}} \,  + \frac{c_\varepsilon}{(t-s)^2}\sum_{i=1}^{n}  \Big(  \int_{B_t} \, \left( \, |u_{x_i}(x+e_jh)|^2+|u_{x_i}(x) |^2 \, \right)^\frac{p_i}{2}  dx \Big)^{\frac{p_i -2}{p_i}} \cdot \notag \\
& \qquad \qquad \qquad \ \ \ \ \ \cdot \Big( \int_{B_t} | \tau_{j,h} u|^{p_i} dx \Big)^{\frac{2}{p_i}}  \notag \\
& \leq 2 \varepsilon \,\textbf{\textbf{RHS}} \,  + \frac{c_\varepsilon}{(t-s)^2}\sum_{i=1}^{n} \Big(  \int_{B_{t'}} \,  \,|u_{x_i}(x) |^{p_i} \,  dx \Big)^{\frac{p_i -2}{p_i}}\cdot \Big( \int_{B_t} | \tau_{j,h} u|^{p_i} dx \Big)^{\frac{2}{p_i}},
\end{align}
where in the last two inequalities we used that $|D \eta| \leq \frac{c}{t-s}$, H\"older's inequality and Lemma \ref{ldiff}. For the estimate of the last term in \eqref{termtau}, we have to distinguish between two cases, as done in \cite{BraTau}.\\
\textbf{Case 1:} $\quad \mathbf{i \leq j \leq n}$. 
In this case we have that $p_i \leq p_j$ and this allows us to apply H\"older's inequality as follows
\begin{align}\label{tau2}
\int_{B_t} |\tau_{j,h} u|^{p_i} \, dx &\leq c(R, p_i, p_j) \left( \int_{B_t} |\tau_{j,h} u|^{p_j} \, dx \right)^{\frac{p_i}{p_j}} \leq c(R, p_i, p_j) |h|^{p_i} \left( \int_{B_{t'}} |u_{x_j}|^{p_j} \, dx \right)^{\frac{p_i}{p_j}}.
\end{align}
\textbf{Case 2:}  $\quad \mathbf{j+1 \leq i \leq n }$. Now, it holds $p_i \geq p_j$, but, by assumption \eqref{Ipotesip}, we know that $p_i \le p_n \le p_1+2  \le p_j +2$. Therefore, by H\"older's inequality we deduce


\begin{align}\label{tau3}
\int_{B_t} |\tau_{j,h} u|^{p_i} \, dx &\leq c(R, p_i, p_j) \left( \int_{B_t} |\tau_{j,h} u|^{p_j+2} \right)^{\frac{p_i}{ p_j+2}}\notag\\ &\leq c(R, p_i, p_j) |h|^{p_i} \left( \int_{B_{t'}} |u_{x_j}|^{p_j+2} \right)^{\frac{p_i}{ p_j+2}},
\end{align}
{where the last integral is finite by \eqref{highint}. }
\\Next, inserting \eqref{tau2} and \eqref{tau3} in \eqref{termtau}, we infer
\begin{align}
 |I_2|  & \leq 2 \varepsilon \,\textbf{\textbf{RHS}} \,  + \frac{c_\varepsilon}{(t-s)^2} \, |h|^2\sum_{i=1}^{j}  \, \Big(  \int_{B_{t'}} \,  \,|u_{x_i}(x) |^{p_i} \,  dx \Big)^{\frac{p_i -2}{p_i}} \cdot \left( \int_{B_{t'}} |u_{x_j}|^{p_j} \, dx  \right)^{\frac{2}{p_j}} \notag \\
& \qquad  + \frac{c_\varepsilon}{(t-s)^2} \, |h|^2\sum_{i=j+1}^{n}  \, \Big(  \int_{B_{t'}} \,  \,|u_{x_i}(x) |^{p_i} \,  dx \Big)^{\frac{p_i -2}{p_i}} \cdot \left( \int_{B_{t'}} |u_{x_j}|^{p_j+2} \right)^{\frac{{2}}{ p_j+2}} \notag \\
& \leq 2 \varepsilon \,\textbf{\textbf{RHS}} \,  +\frac{c_\varepsilon}{(t-s)^2} \, |h|^{2}  \sum_{i=1}^{j}  \, \Big(  \int_{B_{t'}} \,  \,|u_{x_i}(x) |^{p_i} \,  dx \Big)^{\frac{p_i -2}{p_i}} \cdot \left( \int_{B_{t'}} |u_{x_j}|^{p_j} \, dx  \right)^{\frac{2}{p_j}} \notag \\
& \qquad  + \frac{c_{\varepsilon,\beta}}{(t-s)^2} \, |h|^{2} \sum_{i=j+1}^{n}  \, \Big(  \int_{B_{t'}} \,  \,|u_{x_i}(x) |^{p_i} \,  dx \Big)^{{\frac{(p_i -2)(p_j+2)}{p_i\,p_j}}} + \beta |h|^{2} \int_{B_{t'}} |u_{x_j}|^{p_j+2} , \notag 
\end{align}
where, in the last line, we used Young's inequality. Therefore, we conclude that
\begin{align}\label{I2}
    |I_2| & \leq  2 \varepsilon \,\textbf{\textbf{RHS}} \,  +\frac{c_\varepsilon}{(t-s)^2} \, |h|^{2}  \sum_{i=1}^{j}  \, \Big(  \int_{B_{t'}} \,  \,|u_{x_i}(x) |^{p_i} \,  dx \Big)^{\frac{p_i -2}{p_i}} \cdot \left( \int_{B_{t'}} |u_{x_j}|^{p_j} \, dx  \right)^{\frac{2}{p_j}} \notag \\
& \qquad  + \frac{c_{\varepsilon,\beta}}{(t-s)^2} \, |h|^{2} \sum_{i=j+1}^{n}  \, \Big(  \int_{B_{t'}} \,  \,|u_{x_i}(x) |^{p_i} \,  dx \Big)^{{\frac{(p_i -2)(p_j+2)}{p_i\,p_j}}} + \beta |h|^{2} \int_{B_{t'}} |u_{x_j}|^{p_j+2} .
\end{align}
From condition \eqref{A4}, Lemma \ref{ldiff}, the properties of $\eta$ and H\"older's inequality, 
 we derive
\begin{align}\label{I3prima}
    |I_3| &= \left| \int_{\Omega} \langle f_{\xi} (x+e_jh, Du(x))- f_{\xi} (x, Du(x)) , 2 \eta D \eta \tau_{j,h} u \rangle \, dx  \right| \notag \\
    & \leq 2|h| \sum_{i=1}^n \int_{\Omega} \eta | \eta_{x_i}||\tau_{j, h} u| \, \Big(g(x+h)+ g(x) \Big) \,|u_{x_i}|^{p_i-1} dx \notag \\
    & \leq \frac{c}{t-s}|h| \sum_{i=1}^n \int_{B_t}  |\tau_{j, h} u| \, \Big(g(x+h)+ g(x) \Big) \,|u_{x_i}|^{p_i-1} dx \notag \\
      &  \leq  \frac{c}{t-s} \, |h| \, \left( \int_{B_{R}} g(x)^{r} dx \right)^{\frac{1}{r}}  \underbrace{\sum_{i=1}^n\left( \int_{B_{t}}|\tau_{j,h} u|^{\frac{r}{r-1}} |u_{x_i}|^{\frac{r(p_i-1)}{r-1}} dx \right)^{\frac{r-1}{r}}}_{J}.  
\end{align}

Notice that 
$$\frac{1}{p_j+2}\frac{r}{r-1}  < 1$$ 
if and only if
\begin{equation}\label{Stima3p}
  r> \frac{p_j+2}{p_j+1}.
\end{equation}
Inequality \eqref{Stima3p} is satisfied by virtue of assumption \eqref{Ipotesir}, once we observe that
$$r >p_n+2 \geq p_j+2 >\frac{p_j+2}{p_j+1}.$$
Thanks to Hölder's inequality and Lemma \ref{ldiff},  we get
\begin{align}
    J &\leq  \left( \int_{B_{t}}|\tau_{j,h} u|^{p_j+2}  \, dx \right)^{\frac{1}{p_j+2}} \  \sum_{i=1}^{n}\left(   \int_{B_t} |u_{x_i}|^{\frac{(p_j+2)(p_i-1)r}{(r-1)(p_j+2)-r}} \, dx \right)^{\frac{(r-1)(p_j+2)-r}{r(p_j+2)}} \notag 
\end{align} 
Notice that 
$$\frac{(p_j+2)(p_i-1)r}{(r-1)(p_j+2)-r} \leq p_i+2$$
if and only if
\begin{equation}\label{sommabiliR}
r \geq \frac{(p_i+2)(p_j+2)}{3p_j-p_i+4}.
\end{equation}
Inequality \eqref{sommabiliR} is satisfied by virtue of assumption \eqref{Ipotesir} and since
$$\frac{(p_i+2)(p_j+2)}{3p_j-p_i+4} < p_i+2,$$
once we observe that
$$p_i < 2p_j +2$$
holds true by assumption \eqref{Ipotesip}.\\
Hence, we can apply Hölder's inequality and Lemma \ref{ldiff}
to infer
\begin{align}\label{J}
    J &\leq  c(R)\left( \int_{B_{t}}|\tau_{j,h} u|^{p_j+2}  \, dx \right)^{\frac{1}{p_j+2}} \  \sum_{i=1}^{n}\left(   \int_{B_t} |u_{x_i}|^{p_i+2} \, dx \right)^{\frac{p_i-1}{p_i+2}} \notag \\
    & \leq c(R)|h| \left( \int_{B_{t}}|u_{x_j}|^{p_j+2}  \, dx \right)^{\frac{1}{p_j+2}} \  \sum_{i=1}^{n}\left(   \int_{B_t} |u_{x_i}|^{p_i+2} \, dx \right)^{\frac{p_i-1}{p_i+2}}. 
\end{align} 
Now inserting \eqref{J} in \eqref{I3prima}, using twice Young's inequality and denoting by $\alpha= \frac{p_i+2}{3}$ the conjugate of the exponent $\frac{p_i+2}{p_i-1}$ and $\delta$ the conjugate of $\frac{p_j+2}{\alpha}>1$,  we have
\begin{align}\label{I3}
    |I_3| &\leq \frac{c_{ \beta}}{(t-s)^\alpha} \, |h|^{2}  \sum_{i=1}^{n}   \left( \int_{B_{R}} g(x)^{r} \, dx \right)^{\frac{\alpha }{r}}\ \left( \int_{B_{t}}|u_{x_j}|^{p_j+2}  \, dx \right)^{\frac{\alpha}{p_j+2}} \notag \\ 
    & \qquad + |h|^{2} \beta \sum_{i=1}^{n} \int_{B_{t'}} |u_{x_i}|^{p_i+2} \, dx  \notag \\
     &\leq   \frac{\bar{c}_{ \beta}}{(t-s)^{\alpha\delta}} \, |h|^{2}   \left( \int_{B_{R}} g(x)^{r} \, dx \right)^{\frac{\alpha\, \delta }{r}} + |h|^{2} \beta   \int_{B_{t}}|u_{x_j}|^{p_j+2}  \, dx  \notag \\ 
    & \qquad + |h|^{2} \beta \sum_{i=1}^{n} \int_{B_{t'}} |u_{x_i}|^{p_i+2} \, dx \notag \\
      &\leq   \frac{\bar{c}_{ \beta}}{(t-s)^{\alpha\delta}} \, |h|^{2}   \left( \int_{B_{R}} g(x)^{r} \, dx \right)^{\frac{\alpha\, \delta }{r}} + |h|^{2} \beta \sum_{i=1}^{n} \int_{B_{t}}|u_{x_i}|^{p_i+2}  \, dx  \notag \\ 
    & \qquad + |h|^{2} \beta \sum_{i=1}^{n}  \int_{B_{t'}} |u_{x_i}|^{p_i+2} \, dx \notag \\
       &\leq   \frac{\bar{c}_{ \beta}}{(t-s)^{\alpha\delta}}  |h|^{2}  \left( \int_{B_{R}} g(x)^{r} \, dx \right)^{\frac{\alpha\, \delta }{r}}  + |h|^{2} 2\beta \sum_{i=1}^{n} \int_{B_{t'}} |u_{x_i}|^{p_i+2} \, dx . 
\end{align}

Similarly as for the estimate of $I_3$, we obtain
\begin{align}\label{IntI4prima}
     |I_4| &= \left|  \int_{\Omega} \langle f_{\xi} (x+e_jh, Du(x))- f_{\xi} (x, Du(x)) , \eta^2  \tau_{j,h} Du \rangle \, dx \right| \notag \\
    & \leq c |h| \sum_{i=1}^{n}  \int_{\Omega} \eta^2 |\tau_{j, h}u_{x_i}| \, (g(x+h)+ g(x))\, |u_{x_i}|^{p_i -1} dx \notag \\
    & \leq \varepsilon  \underbrace{\int_{\Omega} \eta^2 \sum_{i=1}^{n}\, \left( \, |u_{x_i}(x+e_jh)|^2+|u_{x_i}(x) |^2 \, \right)^\frac{p_i -2}{2}| \tau_{j,h} u_{x_i} |^2 dx}_{\textbf{\textbf{RHS}}}  \notag \\
    & \qquad + c_{\varepsilon} |h|^2 \sum_{i=1}^{n}  \int_{\Omega} \eta^2 |u_{x_i}|^{p_i} \, (g(x+h)+ g(x))^2\,  dx \notag \\
     & \leq \varepsilon \textbf{\textbf{RHS}} + c_{\varepsilon} |h|^{2} \sum_{i=1}^{n}  \left( \int_{B_{R}} g(x)^{r} dx \right)^{\frac{2}{r}} \, \left( \int_{B_{t}} |u_{x_i}|^{\frac{p_i r}{r-2}} dx \right)^{\frac{r-2}{r}}.
\end{align}
We note that
$$\frac{(p_i+2)}{p_ir}(r-2)>1 \Longleftrightarrow{r > p_i+2}$$
which is true by assumption \eqref{Ipotesir}.
Therefore applying H\"older's inequality in \eqref{IntI4prima}, we infer 
\begin{align}\label{I4}
 |I_4| &\leq  \varepsilon \textbf{\textbf{RHS}} +    c_{\varepsilon} |h|^{2} \sum_{i=1}^{n}  \left( \int_{B_{R}} g(x)^{r} dx \right)^{\frac{2}{r}} \,   \left( \int_{B_{t}} |u_{x_i}|^{p_i+2} dx \right)^{\frac{p_i}{p_i +2}} \notag \\
 & \leq \varepsilon \textbf{\textbf{RHS}} +  c_{\varepsilon, \beta} |h|^2 \sum_{i=1}^{n} \left( \int_{B_{R}} g(x)^{r} \, dx \right)^{\frac{p_i+2}{r}} +  |h|^2 \beta \sum_{i=1}^{n}   \int_{B_{{t'}}} |u_{x_i}|^{p_i+2} dx  \notag \\
 & \leq \varepsilon \textbf{\textbf{RHS}} +  c_{\varepsilon, \beta} |h|^2 \left( \int_{B_{R}} g(x)^{r} \, dx \right)^{\gamma} +  |h|^2 \beta \sum_{i=1}^{n}   \int_{B_{{t'}}} |u_{x_i}|^{p_i+2} dx ,
\end{align}
where we used Young's inequality with exponents $ \frac{p_i +2}{p_i}$ and $\frac{p_i+2}{2}$ and $\gamma=\gamma(p_i)$.

By applying H\"older's inequality, Lemma \ref{ldiff} and Young's inequality, we estimate $|I_5|$ as follows
\begin{align}
    |I_5| \le & \int_{B_t}|\tau_{j,h}\omega| \, |\tau_{j,h}u| \, dx \notag\\
    \le &  \left( \int_{B_t}|\tau_{j,h}\omega|^{\frac{p_j+2}{p_j+1}} \, dx\right)^\frac{p_j+1}{p_j+2} \left( \int_{B_t}|\tau_{j,h}u|^{p_j+2} \, dx\right)^\frac{1}{p_j+2} \notag\\
    \le & \, c |h|^2 \left( \int_{B_{t'}}|\omega_{x_j}|^{\frac{p_j+2}{p_j+1}} \, dx\right)^\frac{p_j+1}{p_j+2} \left( \int_{B_{t'}}|u_{x_j}|^{p_j+2} \, dx\right)^\frac{1}{p_j+2} \notag\\
    \le & \, c_\beta |h|^2 \sum_{i=1}^n \int_{B_R}|\omega_{x_i}|^{\frac{p_i+2}{p_i+1}} \, dx +\beta|h|^2 \sum_{i=1}^n  \int_{B_{t'}}|u_{x_i}|^{p_i+2} \, dx
    , \label{J_7}
\end{align}
where we also used the properties of $\eta$ and that $\omega \in  W^{1, \frac{\textbf{p}+2}{\textbf{p}+1}}_{\mathrm{loc}}(\Omega)$.

Inserting \eqref{I1}, \eqref{I2}, \eqref{I3}, \eqref{I4} and \eqref{J_7} in \eqref{SommaInt}, we get
\begin{align}\label{Formulacompleta}
 l &\int_{\Omega} \eta^2 \sum_{i=1}^{n}\, \left( \, |u_{x_i}(x+e_jh)|^2+|u_{x_i}(x) |^2 \, \right)^\frac{p_i -2}{2}| \tau_{j,h} u_{x_i} |^2 dx \notag \\
& \qquad \leq  3 \varepsilon \, \int_{\Omega}  \eta^2 \sum_{i=1}^{n}\, \left( \, |u_{x_i}(x+e_jh)|^2+|u_{x_i}(x) |^2 \, \right)^\frac{p_i -2}{2}| \tau_{j,h} u_{x_i} |^2 dx \notag \\ 
& \qquad +\frac{c_\varepsilon}{(t-s)^2} \, |h|^{2}  \sum_{i=1}^{j}  \, \Big(  \int_{B_{t'}} \,  \,|u_{x_i}(x) |^{p_i} \,  dx \Big)^{\frac{p_i -2}{p_i}} \cdot \left( \int_{B_{t'}} |u_{x_j}|^{p_j} \, dx  \right)^{\frac{2}{p_j}} \notag \\
& \qquad  + \frac{c_{\varepsilon,\beta}}{(t-s)^2} \, |h|^{2} \sum_{i=j+1}^{n}  \, \Big(  \int_{B_{t'}} \,  \,|u_{x_i}(x) |^{p_i} \,  dx \Big)^{{\frac{(p_i -2)(p_j+2)}{p_i\,p_j}}} + {5 \beta} |h|^2 \sum_{i=1}^{n} \int_{B_{t'}} |u_{x_i}|^{p_i+2} \notag \\
& {\qquad + \frac{\bar{c}_{\beta}}{{(t-s)^{\alpha\delta}}} \, |h|^{2} \sum_{i=1}^{n} \left( \int_{B_{R}} g(x)^{r} \, dx \right)^{\frac{\alpha\, \delta }{r}} } \notag \\
     & \qquad +   c_{\varepsilon, \beta} |h|^2  \left( \int_{B_{R}} g(x)^{r} \, dx \right)^{^{\gamma}}\notag\\
     & \qquad +\, c_\beta |h|^2 \sum_{i=1}^n \int_{B_R}|\omega_{x_i}|^{\frac{p_i+2}{p_i+1}}  \, dx.  
\end{align}
Choosing $\varepsilon = \frac{l}{6}$, we can reabsorb the first integral in the right-hand side of \eqref{Formulacompleta} by the left-hand side thus getting
\begin{align}
 &\int_{\Omega} \eta^2 \sum_{i=1}^{n}\, \left( \, |u_{x_i}(x+e_jh)|^2+|u_{x_i}(x) |^2 \, \right)^\frac{p_i -2}{2}| \tau_{j,h} u_{x_i} |^2 dx \notag \\
&  \qquad \leq \frac{c}{(t-s)^2} \, |h|^{2}  \sum_{i=1}^{j}  \, \Big(  \int_{B_{t'}} \,  \,|u_{x_i}(x) |^{p_i} \,  dx \Big)^{\frac{p_i -2}{p_i}} \cdot \left( \int_{B_{t'}} |u_{x_j}|^{p_j} \, dx  \right)^{\frac{2}{p_j}} \notag \\
& \qquad  + \frac{c_{\beta}}{(t-s)^2} \, |h|^{2} \sum_{i=j+1}^{n}  \, \Big(  \int_{B_{t'}} \,  \,|u_{x_i}(x) |^{p_i} \,  dx \Big)^{{\frac{(p_i -2)(p_j+2)}{p_i\,p_j}}} +{5 \beta} |h|^2 \sum_{i=1}^{n}  \int_{B_{t'}} |u_{x_i}|^{p_i+2}\notag \\
& {\qquad + \frac{\bar{c}_{\beta}}{{(t-s)^{\alpha\delta}}} \, |h|^{2} \sum_{i=1}^{n} \left( \int_{B_{R}} g(x)^{r} \, dx \right)^{\frac{\alpha\, \delta }{r}} } \notag \\
     & \qquad +   c_{\varepsilon, \beta} |h|^2  \left( \int_{B_{R}} g(x)^{r} \, dx \right)^{^{\gamma}} \notag\\
     & \qquad +\, c_\beta |h|^2 \sum_{i=1}^n \int_{B_R}|\omega_{x_i}|^{\frac{p_i+2}{p_i+1}}  \, dx.  
\end{align}
Then, by Lemmas \ref{D1} and \ref{Lemmahzero}, letting $|h| \to 0,$ we obtain
\begin{align}\label{sommauxij}
 &    \sum_{i=1}^{n}   \left( \int_{B_s} |u_{x_i x_j}|^2|u_{x_i}|^{p_i-2} dx \right) \notag \\
&  \qquad \leq \frac{c}{(t-s)^2} \,   \sum_{i=1}^{j}  \, \Big(  \int_{B_{R}} \,  \,|u_{x_i}(x) |^{p_i} \,  dx \Big)^{\frac{p_i -2}{p_i}} \cdot \left( \int_{B_{R}} |u_{x_j}|^{p_j} \, dx  \right)^{\frac{2}{p_j}} \notag \\
& \qquad  + \frac{c}{(t-s)^2} \,  \sum_{i=j+1}^{n}  \, \Big(  \int_{B_{R}} \,  \,|u_{x_i}(x) |^{p_i} \,  dx \Big)^{{\frac{(p_i -2)(p_j+2)}{p_i\,p_j}}}  \notag \\
& {\qquad + \frac{\bar{c}_{\beta}}{{(t-s)^{\alpha\delta}}} \, \sum_{i=1}^{n} \left( \int_{B_{R}} g(x)^{r} \, dx \right)^{\frac{\alpha\, \delta }{r}} } \notag \\
     & \qquad +   c_{\varepsilon, \beta}  \left( \int_{B_{R}} g(x)^{r} \, dx \right)^{{\gamma}}
     + {5 \beta} \sum_{i=1}^{n}  \int_{B_{t'}} |u_{x_i}|^{p_i+2}  \notag\\
     & \qquad +\, c_\beta \sum_{i=1}^n \int_{B_R}|\omega_{x_i}|^{\frac{p_i+2}{p_i+1}}  \, dx ,
\end{align}
which  in particular yields for every $i=1, \dots, n$
\begin{align}
 &       \int_{B_s} |u_{x_i x_i}|^2|u_{x_i}|^{p_i-2} dx   \leq \frac{c}{(t-s)^2} \,   \sum_{h=1}^{i}  \, \Big(  \int_{B_{R}} \,  \,|u_{x_h}(x) |^{p_h} \,  dx \Big)^{\frac{p_h -2}{p_h}} \cdot \left( \int_{B_{R}} |u_{x_i}|^{p_i} \, dx  \right)^{\frac{2}{p_i}} \notag \\
& \qquad  + \frac{c}{(t-s)^2} \,  \sum_{h=i+1}^{n}  \, \Big(  \int_{B_{R}} \,  \,|u_{x_h}(x) |^{p_h} \,  dx \Big)^{{\frac{(p_h -2)(p_i+2)}{p_h\,p_i}}}  \notag \\
& {\qquad + \frac{\bar{c}_{\beta}}{{(t-s)^{\alpha\delta}}} \,  \sum_{i=1}^{n} \left( \int_{B_{R}} g(x)^{r} \, dx \right)^{\frac{\alpha\, \delta }{r}} } \notag \\
     & \qquad +   c_{\varepsilon, \beta}   \left( \int_{B_{R}} g(x)^{r} \, dx \right)^{^{\gamma}}
      + {5 \beta} \sum_{h=1}^{n}  \int_{B_{t'}} |u_{x_h}|^{p_h+2}  \notag\\
      & \qquad  +\, c_\beta \sum_{h=1}^n \int_{B_R}|\omega_{x_h}|^{\frac{p_h+2}{p_h+1}}  \, dx .
\end{align}
Summing over $i=1,\dots, n$, the previous inequality, we infer
\begin{align}\label{sommasuj}
     \sum_{i=1}^n & \left(   \int_{B_s} |u_{x_i x_i}|^2|u_{x_i}|^{p_i-2} dx  \right) \notag \\
 &\leq  \frac{c}{(t-s)^2} \,  \sum_{i=1}^{n}\sum_{h=1}^{i} \left(  \int_{B_{R}} \,  \,|u_{x_h}(x) |^{p_h} \,  dx \right)^{\frac{p_h -2}{p_h}} \cdot \left( \int_{B_{R}} |u_{x_i}(x)|^{p_i} \, dx  \right)^{\frac{2}{p_i}} \notag \\
& \qquad  + \frac{c}{(t-s)^2} \,   \sum_{i=1}^{n} \sum_{h=i+1}^{n}  \, \Big(  \int_{B_{R}} \,  \,|u_{x_h}(x) |^{p_h} \,  dx \Big)^{{\frac{(p_h -2)(p_i+2)}{p_h\,p_i}}}   \notag \\
& \qquad +n\Bigg[   { \frac{\bar{c}_{\beta}}{{(t-s)^{\alpha\delta}}} \,  \sum_{h=1}^{n} \left( \int_{B_{R}} g(x)^{r} \, dx \right)^{\frac{\alpha\, \delta }{r}} } \notag \\
     & \qquad +   c_{\varepsilon, \beta}  \left( \int_{B_{R}} g(x)^{r} \, dx \right)^{^{\gamma}}
      + {5 \beta} \sum_{h=1}^{n} \left( \int_{B_{t'}} |u_{x_h}|^{p_h+2} \right)\notag\\
      & \qquad +\, c_\beta \sum_{h=1}^n \int_{B_R}|\omega_{x_h}|^{\frac{p_h+2}{p_h+1}}  \, dx \Bigg] \notag \\
      &  \leq \frac{c}{(t-s)^2} \,   \sum_{i=1}^{n} \left( 1+ \sum_{h=1}^{n}\int_{B_{R}} \,  \,|u_{x_h}(x) |^{p_h} \,  dx \right)^{\frac{p_i -2}{p_i}} \cdot  \sum_{i=1}^{n} \left( 1 + \sum_{h=1}^{n}\int_{B_{R}} |u_{x_h}(x)|^{p_h} \, dx  \right)^{\frac{2}{p_i}} \notag \\
      & \qquad  + \frac{c}{(t-s)^2} \,   \sum_{i=1}^{n} \sum_{h=i+1}^{n}  \, \Big(  \int_{B_{R}} \,  \,|u_{x_h}(x) |^{p_h} \,  dx \Big)^{{\frac{(p_h -2)(p_i+2)}{p_h\,p_i}}}   \notag \\
& \qquad +n\Bigg[  { \frac{\bar{c}_{\beta}}{{(t-s)^{\alpha\delta}}} \,  \sum_{i=1}^{n} \left( \int_{B_{R}} g(x)^{r} \, dx \right)^{\frac{\alpha\, \delta }{r}} } \notag \\
     & \qquad +   c_{\varepsilon, \beta}   \left( \int_{B_{R}} g(x)^{r} \, dx \right)^{^{\gamma}}
      + {5 \beta} \sum_{h=1}^{n} \left( \int_{B_{t'}} |u_{x_h}|^{p_h+2} \right)\notag\\
      & \qquad +\, c_\beta \sum_{h=1}^n \int_{B_R}|\omega_{x_h}|^{\frac{p_h+2}{p_h+1}}  \, dx \Bigg]  \notag \\
        &  \leq \frac{c}{(t-s)^2} \,   \sum_{i=1}^{n} \left( 1+ \sum_{h=1}^{n}\int_{B_{R}} \,  \,|u_{x_h}(x) |^{p_h} \,  dx \right)^{\tilde{\alpha}} \cdot  \sum_{i=1}^{n} \left( 1 + \sum_{h=1}^{n}\int_{B_{R}} |u_{x_h}(x)|^{p_h} \, dx  \right)^{\tilde{\beta}} \notag \\
      & \qquad  + \frac{c}{(t-s)^2} \,   \sum_{i=1}^{n} \sum_{h=i+1}^{n}  \, \Big(  \int_{B_{R}} \,  \,|u_{x_h}(x) |^{p_h} \,  dx \Big)^{{\frac{(p_h -2)(p_i+2)}{p_h\,p_i}}}   \notag \\
& \qquad +n\Bigg[  { \frac{\bar{c}_{\beta}}{{(t-s)^{\alpha\delta}}} \,  \sum_{h=1}^{n} \left( \int_{B_{R}} g(x)^{r} \, dx \right)^{\frac{\alpha\, \delta }{r}} } \notag \\
     & \qquad + c_{\beta}   \left( \int_{B_{R}} g(x)^{r} \, dx \right)^{^{\gamma}}
      + {5 \beta}\sum_{h=1}^{n} \left( \int_{B_{t'}} |u_{x_h}|^{p_h+2} \right) \notag \\
      & \qquad +\, c_\beta \sum_{h=1}^n \int_{B_R}|\omega_{x_h}|^{\frac{p_h+2}{p_h+1}}  \, dx \Bigg]  \notag \\
       &  \leq \frac{c}{(t-s)^2} \,  n^2 \left( 1+ \sum_{h=1}^{n}\int_{B_{R}} \,  \,|u_{x_h}(x) |^{p_h} \,  dx \right)^{\tilde{\alpha}} \cdot \left( 1 + \sum_{h=1}^{n}\int_{B_{R}} |u_{x_h}(x)|^{p_h} \, dx  \right)^{\tilde{\beta}} \notag \\
      & \qquad  + \frac{c}{(t-s)^2} \,   \sum_{i=1}^{n} \sum_{h=i+1}^{n}  \, \Big(  \int_{B_{R}} \,  \,|u_{x_h}(x) |^{p_h} \,  dx \Big)^{{\frac{(p_h -2)(p_i+2)}{p_h\,p_i}}}   \notag \\
& \qquad +n\Bigg[   { \frac{\bar{c}_{\beta}}{{(t-s)^{\alpha\delta}}} \,  \sum_{h=1}^{n} \left( \int_{B_{R}} g(x)^{r} \, dx \right)^{\frac{\alpha\, \delta }{r}} } \notag \\
     & \qquad +   c_{\varepsilon, \beta} |h|^2  \left( \int_{B_{R}} g(x)^{r} \, dx \right)^{^{\gamma}}
      + {5 \beta} \sum_{h=1}^{n} \left( \int_{B_{t'}} |u_{x_h}|^{p_h+2} \right) \notag \\
      & \qquad +\, c_\beta \sum_{h=1}^n \int_{B_R}|\omega_{x_h}|^{\frac{p_h+2}{p_h+1}}  \, dx \Bigg]  \notag \\
       &  \leq \frac{n^2 c}{(t-s)^2} \,   \left( 1+ \sum_{h=1}^{n}\int_{B_{R}} \,  \,|u_{x_h}(x) |^{p_h} \,  dx \right)^{\tilde{\alpha}+ \tilde{\beta}}  \notag \\
      & \qquad  + \frac{c}{(t-s)^2} \,   \sum_{i=1}^{n} \sum_{h=i+1}^{n}  \, \Big(  \int_{B_{R}} \,  \,|u_{x_h}(x) |^{p_h} \,  dx \Big)^{{\frac{(p_h -2)(p_i+2)}{p_h\,p_i}}}   \notag \\
& \qquad +n\Bigg[   { \frac{\bar{c}_{\beta}}{{(t-s)^{\alpha\delta}}} \,  \sum_{h=1}^{n} \left( \int_{B_{R}} g(x)^{r} \, dx \right)^{\frac{\alpha\, \delta }{r}} } \notag \\
     & \qquad +   c_{\varepsilon, \beta} |h|^2  \left( \int_{B_{R}} g(x)^{r} \, dx \right)^{{\gamma}} 
      + {5 \beta} \sum_{h=1}^{n} \left( \int_{B_{t'}} |u_{x_h}|^{p_h+2} \right)\notag\\
      & \qquad +\, c_\beta \sum_{h=1}^n \int_{B_R}|\omega_{x_h}|^{\frac{p_h+2}{p_h+1}}  \, dx \Bigg]  , 
\end{align}
where $\tilde{\alpha}= \max \{ \frac{p_i-2}{p_i}, i=1, \dots, n\}$ and $\tilde{\beta}= \max \{ \frac{2}{p_i}, i=1, \dots, n\}$.\\ 
From the hypothesis
 $$V_{p_i}(u_{x_i}) \in W^{1,2}_{\mathrm{loc}}(\Omega)$$
and Proposition \ref{LemmaBrasco} together Theorem with \ref{ulimitatoCupini}, we have
$$u_{x_i} \in L^{p_i+2}_{\mathrm{loc}}(\Omega ), \quad \forall i= 1, \dots, n $$
and the following estimate holds for every $i=1,\dots, n$
 \begin{align}\label{TesiProprof}
      \int_{B_{{t'}}} |u_{x_i}|^{p_i+2} dx &\leq c \Vert u \Vert_{\infty} \left[   \left( \int_{B_{\sigma}} |u_{x_i x_i}|^2|u_{x_i}|^{p_i-2} dx \right)+ \frac{1}{(\sigma-t')^\gamma}  \left(\int_{B_{\sigma}} |u_{x_i}|^{p_i} dx \right)^{\gamma} \right] \notag \\
      & \leq c \Vert u \Vert_{\infty} \left[   \left( \int_{B_{\sigma}} |u_{x_i x_i}|^2|u_{x_i}|^{p_i-2} dx \right)+ \frac{1}{(\sigma-t')^\gamma}  \left(\int_{B_R} |u_{x_i}|^{p_i} dx \right)^{\gamma} \right],
   \end{align}
   with $c=c(n, p_i)$ and $\gamma=\gamma(p_i).$\\
Summing the inequality \eqref{TesiProprof} over  $i=1,\dots, n$ and inserting the resulting inequality in \eqref{sommasuj}, we find that
\begin{align}\label{uxipi+2}
 &  \sum_{i=1}^n \left(   \int_{B_s} |u_{x_i x_i}|^2|u_{x_i}|^{p_i-2} dx   \right)  \notag \\ 
& \leq    \beta \,\tilde{c}\lVert u \rVert_{\infty} \sum_{i=1}^{n}   \left( \int_{B_{{\sigma}}} |u_{x_i x_i}|^2|u_{x_i}|^{p_i-2} dx \right)+ \frac{n^2 c}{(t-s)^2} \,   \left( 1+ \sum_{i=1}^{n}\int_{B_{R}} \,  \,|u_{x_i}(x) |^{p_i} \,  dx \right)^{\tilde{\alpha}+ \tilde{\beta}}   \notag \\
& \qquad  + \frac{c_2}{(t-s)^2} \,  \sum_{j=1}^{n} \sum_{i=j+1}^{n}  \, \Big(  \int_{B_{R}} \,  \,|u_{x_i}(x) |^{p_i} \,  dx \Big)^{ {\frac{(p_i -2)(p_j+2)}{p_i\,p_j}}}  \notag \\
& \qquad + \frac{c_1}{{(t-s)^{\alpha\delta}}} \, {\, \sum_{i=1}^{n} \left( \int_{B_{R}} g(x)^{r} \, dx \right)^{\frac{\alpha\, \delta }{r}} } \notag \\
     & \qquad + c_{3} \left[ \left( \int_{B_{R}} g(x)^{r} \, dx \right)^{\frac{p_n+2}{r}} {+\,  \sum_{h=1}^n \int_{B_R}|\omega_{x_h}|^{\frac{p_h+2}{p_h+1}}  \, dx}
     \right] \notag \\
    &  \qquad +  \frac{c_4}{({\sigma}-t')^\gamma}\sum_{i=1}^{n}  \left(\int_{B_R} |u_{x_i}|^{p_i} dx \right)^{\gamma}  .
\end{align}
Setting  $$ \phi(\rho)= \sum_{i=1}^n   \left( \int_{B_\rho} |u_{x_i x_i}|^2|u_{x_i}|^{p_i-2} dx \right) dx, $$
inequality \eqref{uxipi+2} can be written as follows
    \begin{equation}\label{phirho}
\phi(\rho) \leq \beta \tilde{c} \lVert u \rVert_{\infty} \phi({\sigma})+ c_3 + \frac{c_2}{(t-s)^2} +\frac{c_1}{{(t-s)^{\alpha\delta}}} +  \frac{c_4}{({\sigma}-t')^\gamma}
\end{equation}
Choosing $\beta >0$ such that $ \theta =\beta \tilde{c} \lVert u \rVert_{\infty}<1$ and since \eqref{phirho} holds true for every radii $\rho<s<t<t'<\sigma$, we may select radii selecting radii
  $$s= \rho + \frac{\sigma-\rho}{4}, \,t = \rho + \frac{\sigma-\rho}{2}, \, t'= \rho + \frac{3(\sigma-\rho)}{4},$$
  so that $t-s=\sigma-t'=\frac{\sigma- \rho}{4}.$ Hence,  \eqref{phirho} becomes
  \begin{equation}\label{phir}
\phi(\rho) \leq \theta  \phi(\sigma)+ c_3 + 
\frac{\tilde{c}_2}{(\sigma-\rho)^2} +\frac{\tilde{c}_1}{(\sigma-\rho)^{\alpha\delta}} +  \frac{\tilde{c}_4}{(\sigma-\rho)^\gamma}.
\end{equation}
Since \eqref{phir} holds true for every $\frac{R}{4}<\rho<\sigma<R$, we are legitimate to apply Lemma \ref{lm3} to the function $\phi: [\frac{R}{4},R] \to \mathbb{R}$, thus getting
\begin{equation}\label{phiR}
    \phi \left(\frac{R}{4} \right) \leq  c_3 + 
    \frac{\tilde{c}_2}{R^2} +\frac{\tilde{c}_1}{R^{\alpha \delta}} +  \frac{\tilde{c}_4}{R^\gamma}.
    \end{equation}
   Recalling the definition of $\phi(t)$, \eqref{phiR} yields that 
    \begin{align}\label{Stimacompattauxipi+2}
        \sum_{i=1}^{n}   \int_{B_{R/4}} |u_{x_i x_i}|^2|u_{x_i}|^{p_i-2} dx  & \leq c  \left( 1+\sum_{i=1}^{n} \left(\Vert u_{x_i} \Vert_{L^{p_i}(B_R)}+{\Vert \omega_{x_i} \Vert_{L^{\frac{p_i+2}{p_i+1}} (B_R) }}\right) + \Vert g \Vert_{L^r (B_{R})} \right)^\sigma ,
\end{align}
where $c=c(p_i,n,l, L, R, \Vert u \Vert_{\infty})$ and $\sigma=\sigma(n,p_i, r)$.
Putting \eqref{Stimacompattauxipi+2} in \eqref{TesiProprof}, we derive also
\begin{align}\label{pi+2}
 \sum_{i=1}^n \int_{B_{R/4}} |u_{x_i}|^{p_i+2} dx & \leq c \left( 1+\sum_{i=1}^{n} \left(\Vert u_{x_i} \Vert_{L^{p_i}(B_R)} +{\Vert \omega_{x_i} \Vert_{L^{\frac{p_i+2}{p_i+1}} (B_R) }}\right)+ \Vert g \Vert_{L^r (B_{R})} \right)^\sigma
     , 
\end{align}
for positive constants $c = c(n,p_i,\lambda, \Lambda,R, \Vert u \Vert_{\infty})$ and $\sigma= \sigma (n,p_i, r)$.\\
Now, inserting \eqref{pi+2} in \eqref{sommauxij}, we infer for every $j=1, \dots, n$
\begin{align}\label{Stimacompattauxixjpi+2}
        \sum_{i=1}^{n}   \left( \int_{B_{R/4}} |u_{x_i x_j}|^2|u_{x_i}|^{p_i-2} dx \right) & \leq c  \left(1+ \sum_{i=1}^{n} \left(\Vert u_{x_i} \Vert_{L^{p_i}(B_R)}+{\Vert \omega_{x_i} \Vert_{L^{\frac{p_i+2}{p_i+1}} (B_R) }} \right) + \Vert g \Vert_{L^r (B_{R})} \right)^\sigma,
\end{align}
which concludes the proof.
\end{proof}
\section{Approximation}\label{Appro}

In order to perform the approximation argument, let us fix a non-negative smooth kernel $\phi \in \mathcal{C}^\infty_0(B_1(0))$ such that $\int_{B_1(0)} \phi =1$ and consider the corresponding family of mollifiers $(\phi_k)_{k >0}$. We set, for a given ball $B_R \Subset \Omega$ ,
$$a_i^k=a_i * \phi_k, \quad \omega^k= \omega * \phi_k,$$
\begin{equation}\label{gk}
g_k(x)=\underset{i}{\max} \{ |D a_i^k (x)|\}
\end{equation}
and, for $x \in B_R$, we define
\begin{equation}\label{fvarepsilon}
f^k (x,\xi)= \sum_{i=1}^{n} a_i^k(x)|\xi|^{p_i} ,
\end{equation}
 for every $k < \text{dist}(B_R,\Omega)$. We shall use the following $$\mathcal{F}^{k}(v,B_R):=\int_{_{B_R} }\left( f^{k}(x, Dv)-\omega(x)v(x) \right) dx$$ and
$$\mathcal{F}^{k}_{\varepsilon}(v,B_R):= \, \int_{B_R} \left( f^k(x,Dv)+ \varepsilon(1+ |Dv|^2)^{\frac{p_n}{2}}- \omega^k(x)v(x) \right) dx.$$
\begin{flushleft}
One can easily check that $f^k(x, \xi)$ satisfies assumptions \eqref{A2}--\eqref{A3} and \eqref{A'4}, with $K = \sup_{B_R} |g_k(x)| $ that is finite since $a^k_i(x) \in W^{1, \infty}(B_R)$ and then $g_k \in L^{\infty}(B_R)$.\\
\end{flushleft}

We are in position to prove our main result.

\begin{proof}[\textit{{Proof of Theorem~\ref{thmBfinito}}}]
We consider the variational problems
\begin{equation}\label{Pfj}
    \inf \left\{  \int_{B_R} \left( f^k(x,Dv)+ \varepsilon(1+ |Dv|^2)^{\frac{p_n}{2}} -\omega^k(x)v(x)\right) dx \ : \ v \in W^{1,p_n}_0(B_R)+u^\eta \right\},
\end{equation}
where $$u^\eta = u * \varsigma_\eta$$ is the mollification of the local minimizer $u$ of \eqref{functional}, for a sequence of mollifier $\varsigma_n$.\\
It is well known that, by the direct methods of the calculus of variations, there exists a unique solution $u^\eta_{k,\varepsilon} \in  W^{1,p_n}_0(B_R)+u^\eta $ of the problem \eqref{Pfj}. Since the integrand $$f^k(x,Dv)+ \varepsilon(1+ |Dv|^2)^{\frac{p_n}{2}}-\omega^k(x)v(x)$$ satisfies the assumptions of  Theorem \ref{thm2Mascolo}, we have that $$V_{p_i}((u^\eta_{k,\varepsilon})_{x_i}) \in W_{\mathrm{loc}}^{1,2}(B_R) \quad \forall i=1,...,n.$$
Hence we are legitimate to use estimate \eqref{uxistima} of Theorem \ref{AppThm}, to obtain that
\begin{align}
   & \sum_{i=1}^{n}  \int_{B_{R/4}}|(u^\eta_{k,\varepsilon})_{x_i}|^{p_i +2 } dx \notag \\
    & \leq c \left(1+ \sum_{i=1}^{n} \left(\int_{B_R}\, \lvert(u^\eta_{k,\varepsilon})_{x_i} \rvert^{p_i} \ dx+  \Vert \omega^k_{x_i} \Vert_{L^\frac{p_i+2}{p_i+1}(B_{R})} \right)  + \Vert g_k \Vert_{L^r (B_{R})} \right)^\sigma \notag \\
     &  \leq c(\lambda) \left(1+ \sum_{i=1}^{n} \left( \int_{B_R} a_i^k(x)\, \lvert (u^\eta_{k,\varepsilon})_{x_i}\rvert^{p_i} \ dx +  \Vert \omega^k_{x_i} \Vert_{L^\frac{p_i+2}{p_i+1}(B_{R})} \right) + \Vert g_k \Vert_{L^r (B_{R})} \right)^\sigma, \label{der}
\end{align}
where  $c$ is a positive constant that is independent on $\varepsilon, k$ and $\eta$ and where  we used that $a_i^k(x)\geq \lambda, \forall i,k$. 
Now, by the very definition of $f^k$, the minimality of $u^\eta_{k,\varepsilon}$ and using $u^\eta$ as test function, we get
\begin{align}
    \sum_{i=1}^{n} & \int_{B_R}\, a_i^k(x) \lvert(u^\eta_{k,\varepsilon})_{x_i} \rvert^{p_i} \ dx \notag\\
    = & \int_{B_R} f^k(x,Du^\eta_{k,\varepsilon}) \ dx \notag\\
    \le & \int_{B_R} \left(f^k(x,Du^\eta_{k,\varepsilon}) +\varepsilon(1+|Du^\eta_{k,\varepsilon}|^2)^{\frac{p_n}{2}}  -\omega(x)u^\eta_{k,\varepsilon} \right) dx+  \int_{B_R}\omega(x)u^\eta_{k,\varepsilon} dx \notag\\
    \le & \int_{B_R} \left(f^k(x,Du^\eta) +\varepsilon(1+|Du^\eta|^2)^{\frac{p_n}{2}}  -\omega(x)u^\eta \right) dx+  \int_{B_R}\omega(x)u^\eta_{k,\varepsilon} dx \notag\\
    = & \int_{B_R} \left(f^k(x,Du^\eta) +\varepsilon(1+|Du^\eta|^2)^{\frac{p_n}{2}}   \right) dx+  \int_{B_R}\omega^k(x)(u^\eta_{k,\varepsilon}-u^\eta) dx. \label{Poincare}
\end{align}
From the Sobolev-Poincaré inequality at Lemma \ref{anisopoincare}, we derive that
\begin{align}
    \Vert u^\eta_{k,\varepsilon} -u^\eta\Vert_{L^{\overline{p}^*}(B_R)} &\leq c \sum_{i=1}^n\Vert (u^\eta_{k,\varepsilon})_{x_i} - (u^\eta)_{x_i} \Vert_{L^{p_i}(B_R)} \notag\\
    & \le  c \sum_{i=1}^n\left(\Vert (u^\eta_{k,\varepsilon})_{x_i}  \Vert_{L^{p_i}(B_R)} + \Vert (u^\eta)_{x_i} \Vert_{L^{p_i}(B_R)} \right), \notag
\end{align}
and so the second integral in the right hand side of \eqref{Poincare} can be estimated as follows
\begin{align}\label{piccione}
    \int_{B_R}\omega^k(x)(u^\eta_{k,\varepsilon}-u^\eta) dx \le & \ \Vert \omega^k \Vert_{L^\frac{\overline{p}^*}{\overline{p}^*-1}(B_R)}  \Vert u^\eta_{k,\varepsilon} -u^\eta\Vert_{L^{\overline{p}^*}(B_R)}  \notag\\
    \le & \ c(n,p_i,R) \Vert \omega^k \Vert_{L^\frac{\overline{p}^*}{\overline{p}^*-1}(B_R)} \sum_{i=1}^n\left(\Vert (u^\eta_{k,\varepsilon})_{x_i}  \Vert_{L^{p_i}(B_R)} + \Vert (u^\eta)_{x_i} \Vert_{L^{p_i}(B_R)} \right) \notag\\
    \le & \ \dfrac{\lambda}{2} \sum_{i=1}^n \Vert (u^\eta_{k,\varepsilon})_{x_i} \Vert^{p_i}_{L^{p_i}(B_R)} \notag\\
    & \qquad+ c(n,p_i,\lambda,R) \left( \Vert \omega^k\Vert^{\frac{p_i}{p_i-1}}_{L^\frac{\overline{p}^*}{\overline{p}^*-1}(B_R)}+ \Vert (u^\eta)_{x_i} \Vert^{\frac{p_i}{p_i-1}}_{L^{p_i}(B_R)}+\right)\notag\\
    \le & \ \dfrac{1}{2}  \sum_{i=1}^{n} \int_{B_R}\, a_i^k(x) \lvert(u^\eta_{k,\varepsilon})_{x_i} \rvert^{p_i} \ dx  \notag\\
    & \qquad+ \sum_{i=1}^{n} c(n,p_i,\lambda,R) \left( \Vert \omega^k\Vert^{\frac{p_i}{p_i-1}}_{L^\frac{\overline{p}^*}{\overline{p}^*-1}(B_R)}+ \Vert (u^\eta)_{x_i} \Vert^{\frac{p_i}{p_i-1}}_{L^{p_i}(B_R)} \right) \notag\\
    \le & \ \dfrac{1}{2}  \sum_{i=1}^{n} \int_{B_R}\, a_i^k(x) \lvert(u^\eta_{k,\varepsilon})_{x_i} \rvert^{p_i} \ dx  \notag\\
    & \qquad+ \sum_{i=1}^{n} c(n,p_i,\lambda,R) \left( \Vert \omega^k\Vert^{\frac{p_i}{p_i-1}}_{L^t(B_R)}+ \Vert (u^\eta)_{x_i} \Vert^{\frac{p_i}{p_i-1}}_{L^{p_i}(B_R)} \right)
    , 
\end{align}
where we also used H\"older's and Young's inequalities and that $\omega \in L^{t}(B_R) \subset L^\frac{\overline{p}^*}{\overline{p}^*-1}(B_R)$. Putting the previous estimate in \eqref{Poincare} and reabsorbing the term $\sum_{i=1}^{n} \int_{B_R}\, a_i^k(x) \lvert(u^\eta_{k,\varepsilon})_{x_i} \rvert^{p_i} \ dx $ from the right hand side by the left hand side, we get
\begin{align*}
    \sum_{i=1}^{n} & \int_{B_R}\, a_i^k(x)\lvert(u^\eta_{k,\varepsilon})_{x_i} \rvert^{p_i} \ dx \notag\\
    \le & \, 2\int_{B_R} \left(f^k(x,Du^\eta) +\varepsilon(1+|Du^\eta|^2)^{\frac{p_n}{2}}   \right) dx\notag\\
    & + \sum_{i=1}^{n} c(n,p_i,\lambda,R) \left( \Vert \omega^k\Vert^{\frac{p_i}{p_i-1}}_{L^t(B_R)}+ \Vert (u^\eta)_{x_i} \Vert^{\frac{p_i}{p_i-1}}_{L^{p_i}(B_R)} \right).
\end{align*}
Inserting the previous inequality in \eqref{der}, we obtain
\begin{align}
     & \sum_{i=1}^{n}  \int_{B_{R/4}}|(u^\eta_{k,\varepsilon})_{x_i}|^{p_i +2 } dx \notag \\
     &  \leq c \Bigg(1+ \int_{B_R} \left(f^k(x,Du^\eta) +\varepsilon(1+|Du^\eta|^2)^{\frac{p_n}{2}}   \right) dx + \Vert g_k \Vert_{L^r (B_{R})}  \notag\\
     & \qquad\quad +\sum_{i=1}^{n} \left(\Vert \omega^k\Vert^{\frac{p_i}{p_i-1}}_{L^t(B_R)}+  {\Vert \omega^k_{x_i} \Vert_{L^\frac{p_i+2}{p_i+1}(B_{R})}} + \Vert (u^\eta)_{x_i} \Vert^{\frac{p_i}{p_i-1}}_{L^{p_i}(B_R)} \right)\Bigg)^\sigma,
\end{align}
for a positive constant $c$ independent of $k$, $\varepsilon$ ans $\eta$. 
Since $u^\eta \in W^{1,p_n}(B_R)$, the right-hand side of the previous estimate is finite, and therefore, the left-hand side of the previous estimate is uniformly bounded  with respect to $\varepsilon$, for $\varepsilon \in (0,1)$. Therefore there exists $v^\eta_k$ such that, up to a subsequence,
    $$u^\eta_{k,\varepsilon} \rightharpoonup v^\eta_k \text{  weakly in  } W^{1, \textbf{p}+2}(B_R) \, \text{ as  } \varepsilon \to 0.$$
By the weak lower semicontinuity of the norm, we get
{\begin{align}\label{stimaconv1}
    &\sum_{i=1}^{n}   \int_{B_{R/4}}|(v^\eta_k)_{x_i}|^{p_i +2 } dx  \notag\\
    &\leq \liminf_{\varepsilon}\sum_{i=1}^{n}   \int_{B_{R/4}}|(u^\eta_{k,\varepsilon})_{x_i}|^{p_i +2 } dx  \notag \\
    &  \le c \Bigg(1+ \int_{B_R} f^k(x,Du^\eta)  dx + \Vert g_k \Vert_{L^r (B_{R})}  \notag\\
     & \qquad \quad + \sum_{i=1}^{n}  \left({\Vert \omega^k_{x_i} \Vert_{L^\frac{p_i+2}{p_i+1}(B_{R})}} +\Vert \omega^k\Vert^{\frac{p_i}{p_i-1}}_{L^r(B_R)}+ \Vert (u^\eta)_{x_i} \Vert^{\frac{p_i}{p_i-1}}_{L^{p_i}(B_R)} \right)\Bigg)^\sigma,
\end{align}}
where we used that, since $u^\eta \in W^{1, p_n}(\Omega)$, it holds 
\begin{equation}\label{fkappalim}
\liminf_{\varepsilon \to 0} \, \varepsilon \int_{B_R}(1+|Du^\eta|^2)^{\frac{p_n}{2}} dx =0. 
\end{equation}
 Using the definition of $f_k$ at \eqref{fvarepsilon}, we have that
\begin{equation}\label{convFk}
    \left|  \int_{B_R}   f^k(x,Du^\eta) -  f(x,Du^\eta) dx \right|\leq \sum_{i=1}^n \int_{B_R} |a^k_i(x)- a_i(x)|\, | u^\eta_{x_i}|^{p_i} dx .
\end{equation}
 So the fact that 
 $$a^k_i(x) \to a_i(x)\quad \text{uniformly on compact sets }$$ gives that
$$\lim_{k \to 0^+} \sum_{i=1}^n \lvert \lvert a^k_i(x)- a_i(x) \rvert \rvert_{\infty} \int_{B_R} \, | u^{\eta}_{x_i}|^{p_i} dx =0. $$
Therefore, passing to the limit as $k \to 0^+$ in \eqref{convFk}, yields
\begin{equation}\label{fktoF}
   \int_{B_R} f^k(x, Du^\eta) dx \to \int_{B_R}f(x,Du^\eta) dx.
    \end{equation}
By the properties of mollification it holds $g_k \to g$ in $L^r(B_R)$, where $g_k$ has been found at \eqref{gk} and $g$ is the function appearing in \eqref{A4}. So, by \eqref{fktoF}, the right-hand side of \eqref{stimaconv1} is uniformly bounded w.r.t.\ to $k$. Therefore there exists $v^\eta$ such that, up to a subsequence,
    $$v^\eta_{k} \rightharpoonup v^\eta \text{  weakly in  } W^{1, \textbf{p}+2}(B_R) \, \text{ as  } k \to 0.$$
Again by the weak lower semicontinuity of the norm, \eqref{stimaconv1}, \eqref{fktoF} and the property of mollification, we get
{\begin{align}\label{stimaconv2}
    &\sum_{i=1}^{n} \left(  \int_{B_{R/4}}|(v^\eta)_{x_i}|^{p_i +2 } dx \right)\notag\\
    &\leq \liminf_{k}\sum_{i=1}^{n} \left(  \int_{B_{R/4}}|(v^\eta_{k})_{x_i}|^{p_i +2 } dx \right) \notag \\
    &  \le c \Bigg( 1+\int_{B_R} f(x,Du^\eta)  dx + \Vert g \Vert_{L^r (B_{R})}  \notag\\
     & \qquad \quad +  \sum_{i=1}^{n} \left( {\Vert \omega_{x_i} \Vert_{L^\frac{p_i+2}{p_i+1}(B_{R})}} +\Vert \omega\Vert^{\frac{p_i}{p_i-1}}_{L^\frac{\overline{p}^*}{\overline{p}^*-1}(B_R)}+ \Vert (u^\eta)_{x_i} \Vert^{\frac{p_i}{p_i-1}}_{L^{p_i}(B_R)} \right)\Bigg)^\sigma.
\end{align}}
Now, we note that
\begin{align}\label{convF}
    \left|  \int_{B_R}   f(x,Du^\eta) -  f(x,Du) dx \right| &\leq \left| \sum_{i=1}^n \int_{B_R} \Big( a_i(x)\,  |u^\eta_{x_i}|^{p_i}-a_i(x)  |u_{x_i}|^{p_i} \Big)dx \right| \notag \\
     &\leq  \sum_{i=1}^n \int_{B_R}a_i(x)\,\Big|  |u^\eta_{x_i}|^{p_i} -|u_{x_i}|^{p_i} \Big|dx \notag \\
    &\leq \overline{c } \sum_{i=1}^n \int_{B_R}a_i(x)\,\left(  |u^\eta_{x_i}| +|u_{x_i}| \right)^{p_i-1} \,  |u^\eta_{x_i}- u_{x_i}| dx \notag \\
    &\leq \sum_{i=1}^n  c(\Lambda, p_i) \,\int_{B_R} \left(  |u^\eta_{x_i}| +|u_{x_i}| \right)^{p_i-1} \,  |u^\eta_{x_i}- u_{x_i}| dx  ,
\end{align}
where we used the left inequality in the  Lemma \ref{Duzaar}, with $\xi=|u_{x_i}|$ and $\eta=|u_{x_i}^\eta|$ and $\alpha= p_i$.\\
{Since, by property of mollification, it holds
and
$$ u^\eta_{x_i} \to u_{x_i}  \text{ strongly in }L^{p_i},$$}
passing to the limit as $\eta \to 0^+$ in \eqref{convF}, we have
\begin{equation}\label{IntfDuni}
    \int_{B_R}f(x, Du^\eta) \, dx\to  \int_{B_R}f(x,Du) \, dx .
    \end{equation}
 By virtue of \eqref{IntfDuni} and the fact that $u^\eta \to u$ strongly in $L^{p_i}$,  we observe that the right-hand side of \eqref{stimaconv2} is uniformly bounded w.r.t.\ $\eta$. Therefore there exists $v$ such that, up to a subsequence,
    $$v^\eta \rightharpoonup v \text{  weakly in  } W^{1, \textbf{p}+2}(B_R) \, \text{ as  } \eta \to 0.$$
Once again, by the weak lower semicontinuity of the norm, we derive
\begin{align}\label{stimaconv3}
    &\sum_{i=1}^{n} \left(  \int_{B_{R/4}}|v_{x_i}|^{p_i +2 } dx \right)  \notag\\
    &\leq \liminf_{\eta}\sum_{i=1}^{n} \left(  \int_{B_{R/4}}|(v^\eta)_{x_i}|^{p_i +2 } dx \right) \notag \\
    & \leq   c \liminf_{\eta} \Bigg( 1+\int_{B_R} \left(f(x,Du^\eta) \right) dx + \Vert g \Vert_{L^r (B_{R})}  \notag\\
     & \qquad \quad + \sum_{i=1}^{n}  \left({\Vert \omega_{x_i} \Vert_{L^\frac{p_i+2}{p_i+1}(B_{R})}} +\Vert \omega\Vert^{\frac{p_i}{p_i-1}}_{L^t(B_R)}+ \Vert (u^\eta)_{x_i} \Vert^{\frac{p_i}{p_i-1}}_{L^{p_i}(B_R)} \right)\Bigg)^\sigma \notag \\
    &  = c \Bigg( 1+\int_{B_R} \left(f(x,Du)\right) dx + \Vert g \Vert_{L^r (B_{R})}  \notag\\
     & \qquad \quad + \sum_{i=1}^{n} \left( {\Vert \omega_{x_i} \Vert_{L^\frac{p_i+2}{p_i+1}(B_{R})}} +\Vert \omega\Vert^{\frac{p_i}{p_i-1}}_{L^t(B_R)}+ \Vert u_{x_i} \Vert^{\frac{p_i}{p_i-1}}_{L^{p_i}(B_R)} \right)\Bigg)^\sigma.
\end{align}
In order to conclude the proof, it suffices to show that $u=v$ a.e. in $B_R$.\\
By the weak lower semicontinuity of the functionals  $\mathcal{F}$ and $\mathcal{F}^k$ in $W^{1, \textbf{p}}(B_R)$ , the weak convergence of $v^\eta \rightharpoonup v$ in $W^{1, \textbf{p}+2}(B_R)$, \eqref{fktoF} and $v^\eta_k \rightharpoonup v^\eta$ in $W^{1, \textbf{p}+2}(B_R)$, we get
\begin{align}
    &\int_{B_R}\left(f(x, Dv)-\omega(x)v \right) dx \notag\\
    & \leq \liminf_{\eta} \int_{B_R} \left(f(x, Dv^\eta) -\omega(x)v^\eta \right)dx \notag\\
    &\leq \liminf_{\eta} \liminf_{k} \int_{B_R} \left(f^k(x, Dv^\eta)-\omega(x)v^\eta \right) dx \notag \\   
    & \leq \liminf_{\eta}\liminf_{k} \int_{B_R} \left( f^k(x, Dv^\eta_k)-\omega(x)v^\eta_k \right) dx \notag \\
     & \leq \liminf_{\eta}\liminf_{k} \int_{B_R} \left( f^k(x, Dv^\eta_k) )+ \varepsilon(1+ |Du^\eta_{k, \varepsilon}|^2)^{\frac{p_n}{2}} -\omega(x)v^\eta_k  \right)dx, \notag 
    \end{align}
   where the last inequality is trivial, by the non-negativity of the last term.\\
    Thanks to the weak convergence of $u^\eta_{k, \varepsilon} \rightharpoonup v^\eta_k$ in $W^{1, \textbf{p}+2}(B_R)$ and again the weak lower semicontinuity of functional $\mathcal{F}^k_\varepsilon$, we have 
   \begin{align}\label{pallino}
    &\int_{B_R}\left(f(x, Dv)-\omega(x)v \right)  dx \notag\\
    &\leq \liminf_{\eta}\liminf_{k}\liminf_{\varepsilon} \int_{B_R} \left( f^k(x, Du^\eta_{k, \varepsilon})+ \varepsilon(1+ |Du^\eta_{k, \varepsilon}|^2)^{\frac{p_n}{2}}  -\omega(x)u^\eta_{k,\varepsilon}\right) dx \notag \\
     & \leq \liminf_{\eta}\liminf_{k}\liminf_{\varepsilon} \int_{B_R} \left( f^k(x, Du^\eta)+ \varepsilon(1+ |Du^\eta|^2)^{\frac{p_n}{2}}  -\omega(x)u^\eta\right)dx \notag \\
      & = \liminf_{\eta}\liminf_{k} \int_{B_R}\left( f^k(x, Du^\eta)  -\omega(x)u^\eta \right)dx , 
      \end{align} 
      where we used the minimality of $u^\eta_{k, \varepsilon}$ and \eqref{fkappalim}.\\
      By \eqref{pallino},  \eqref{fktoF} and \eqref{IntfDuni}, we derive
      {\begin{align}
    &\int_{B_R}\left(f(x, Dv)-\omega(x)v \right) dx\notag\\
    &\leq \liminf_{\eta} \int_{B_R} \left(f(x, Du^\eta)- \omega(x)u^\eta \right) dx \notag\\
    &\leq  \int_{B_R} \left( f(x, Du)-\omega(x)u \right) dx,
\end{align}}
and then, by minimality of  $u$,  we infer
\begin{equation*}
     \int_{B_R} \left(f(x, Dv)-\omega(x)v \right) dx 
     =\int_{B_R}  \left( f(x, Du)-\omega(x)u \right) dx.
\end{equation*}
By the strict convexity of $\xi \to f(x, \xi)$, we deduce that $u=v$. Then $u \in W^{1, \textbf{p}+2}$ and so we can argue as in the proof of Theorem \ref{AppThm}, thus obtaining 
\begin{align}
       &\sum_{i=1}^{n}   \left( \int_{B_{\frac{R}{4}}} |u_{x_i x_j}|^2|u_{x_i}|^{p_i-2} dx \right)\notag\\
       &  \le c \Bigg( 1+\int_{B_R} \left(f(x,Du)  \right) dx + \Vert g \Vert_{L^r (B_{R})}  \notag\\
     & \qquad \quad  +\sum_{i=1}^{n} \left( {\Vert \omega_{x_i} \Vert_{L^\frac{p_i+2}{p_i+1}(B_{R})}} +\Vert \omega\Vert^{\frac{p_i}{p_i-1}}_{L^t(B_R)}+ \Vert u_{x_i} \Vert^{\frac{p_i}{p_i-1}}_{L^{p_i}(B_R)}\right)\Bigg)^\sigma \notag\\
     &  \le c \Bigg(  1+   \sum_{i=1}^{n} \left( {\Vert \omega_{x_i} \Vert_{L^\frac{p_i+2}{p_i+1}(B_{R})}} +\Vert \omega\Vert_{L^t(B_R)}+ \Vert u_{x_i} \Vert_{L^{p_i}(B_R)}\right)+ \Vert g \Vert_{L^r (B_{R})}\Bigg)^{\tilde{\sigma}}
\end{align}
and this conclude the proof.
\end{proof}
}

\vskip0.5cm

\noindent {{\bf Acknowledgements.} The authors are members of the Gruppo Nazionale per l'Analisi Matematica, la Probabilità e le loro Applicazioni (GNAMPA) of the Istituto Nazionale di Alta Matematica (INdAM). A.G.\ Grimaldi and S.\ Russo have been supported through the INdAM - GNAMPA 2025 Project "Regolarità di soluzioni di equazioni paraboliche a crescita nonstandard degeneri" (CUP: E5324001950001).
In addition S.\ Russo has also been supported through the project: Sustainable Mobility Center (Centro Nazionale per la Mobilità Sostenibile – CNMS) - SPOKE 10. This manuscript reflects only the author's views and opinions, and the Ministry cannot be considered responsible for them.}\\

\vspace{0,5cm}

\noindent{On behalf of all authors, the corresponding author states that there is no conflict of interest.}

\end{document}